\documentclass[final,3p,12pt]{elsarticle}
\usepackage{amsmath,amsthm,amscd,amssymb,latexsym,upref,stmaryrd}
\usepackage{mathtools}

\makeatletter
\def\ps@pprintTitle{
  \let\@oddhead\@empty
  \let\@evenhead\@empty
  \let\@oddfoot\@empty
  \let\@evenfoot\@oddfoot
}
\makeatother

\numberwithin{equation}{section}

\newtheorem{theorem}{Theorem}[section]
\newtheorem{proposition}[theorem]{Proposition}
\newtheorem{corollary}[theorem]{Corollary}
\newtheorem{lemma}[theorem]{Lemma}
\newtheorem{remark}[theorem]{Remark}
\newtheorem{definition}[theorem]{Definition}

\usepackage{color,soul}

\DeclareMathOperator{\col}{col}

\DeclareMathOperator{\card}{card}

\DeclareMathOperator{\diag}{diag}

\DeclareMathOperator{\dom}{dom}

\DeclareMathOperator{\nlim}{nlim}

\let\Im\relax
\DeclareMathOperator{\Im}{Im}
\let\Re\relax
\DeclareMathOperator{\Re}{Re}

\renewcommand{\le}{\leqslant}
\renewcommand{\ge}{\geqslant}

\newcommand{\len}{1}
\newcommand{\ol}{\overline}
\newcommand{\wt}{\widetilde}
\newcommand{\wh}{\widehat}
\renewcommand{\(}{\left(}
\renewcommand{\)}{\right)}
\renewcommand{\[}{\left[}
\renewcommand{\]}{\right]}
\newcommand{\scal}[1]{\left\langle{#1}\right\rangle}

\newcommand{\eps}{\varepsilon}
\newcommand{\alp}{\alpha}
\renewcommand{\l}{\lambda}
\renewcommand{\L}{\Lambda}
\renewcommand{\phi}{\varphi}

\def\cB{\mathcal{B}}

\def\cI{\mathcal{I}}
\def\cJ{\mathcal{J}}
\def\cK{\mathcal{K}}
\def\cL{\mathcal{L}}
\def\cM{\mathcal{M}}

\def\cR{\mathcal{R}}

\def\cV{\mathcal{V}}

\def\fF{\mathfrak{F}}
\def\fG{\mathfrak{G}}
\def\fH{\mathfrak{H}}

\def\bC{\mathbb{C}}
\def\bD{\mathbb{D}}
\def\bN{\mathbb{N}}
\def\bQ{\mathbb{Q}}
\def\bR{\mathbb{R}}
\def\bT{\mathbb{T}}

\def\bZ{\mathbb{Z}}

\newcommand{\VectorSpace}[2]{{#1}({#2}; \bC^2)}
\newcommand{\MatrixSpace}[2]{{#1}({#2}; \bC^{2 \times 2})}
\newcommand{\LL}[1]{\MatrixSpace{L^{#1}}{[0,1]}}
\newcommand{\LLV}[1]{\VectorSpace{L^{#1}}{[0,1]}}

\newcommand{\bigabs}[1]{\bigl|{#1}\bigr|}
\newcommand{\abs}[1]{\left|{#1}\right|}

\newcommand{\floor}[1]{\left\lfloor{#1}\right\rfloor}
\newcommand{\curl}[1]{\left\{{#1}\right\}}
\newcommand{\fr}[1]{\left\langle{#1}\right\rangle}

\begin{document}

\sloppy

\begin{frontmatter}

\title{Criterion of Bari basis property \\
for $2 \times 2$ Dirac-type operators \\
with strictly regular boundary conditions}

\author{Anton~A.~Lunyov}
\ead{A.A.Lunyov@gmail.com}
\address{
Facebook, Inc.,
1 Hacker Way, Menlo Park,
California, 94025,
United States of America}

\begin{abstract}
The paper is concerned with the Bari basis property of a boundary value problem associated in $L^2([0,1]; \mathbb{C}^2)$ with the following $2 \times 2$ Dirac-type equation for $y = \col(y_1, y_2)$:
$$
 L_U(Q) y = -i B^{-1} y' + Q(x) y = \lambda y , \quad
 B = \begin{pmatrix} b_1 & 0 \\ 0 & b_2 \end{pmatrix}, \quad b_1 < 0 < b_2,
$$
with a potential matrix $Q \in L^2([0,1]; \mathbb{C}^{2 \times 2})$ and subject to the strictly regular boundary conditions $Uy :=\{U_1, U_2\}y=0$. If $b_2 = -b_1 =1$ this equation is equivalent to one dimensional Dirac equation.
We show that the system of root vectors $\{f_n\}_{n \in \bZ}$ of the operator $L_U(Q)$ \emph{forms a Bari basis in} $L^2([0,1]; \mathbb{C}^2)$ if and only if the unperturbed operator $L_U(0)$ is self-adjoint. We also give explicit conditions for this in terms of coefficients in the boundary conditions.

The Bari basis criterion is a consequence of our more general result:
Let $Q \in L^p([0,1]; \mathbb{C}^{2 \times 2})$, $p \in [1,2]$, boundary conditions be strictly regular, and let $\{g_n\}_{n \in \bZ}$ be the sequence biorthogonal to the system of root vectors $\{f_n\}_{n \in \bZ}$ of the operator $L_U(Q)$. Then
$$
\{\|f_n - g_n\|_2\}_{n \in \bZ} \in (\ell^p(\bZ))^*
\quad\Leftrightarrow\quad
L_U(0) = L_U(0)^*.
$$

These abstract results are applied to non-canonical initial-boundary value problem for a damped string equation.
\end{abstract}


\end{frontmatter}

\renewcommand{\contentsname}{Contents}
\tableofcontents

\section{Introduction} \label{sec:Intro}
Consider the following first order system of differential equations
\begin{equation} \label{eq:system}
 \cL y = -i B^{-1} y' + Q(x) y=\l y, \qquad y = \col(y_1,y_2),
 \qquad x \in [0,1],
\end{equation}
where
\begin{equation} \label{eq:BQ}
 B = \begin{pmatrix} b_1 & 0 \\ 0 & b_2 \end{pmatrix},
 \quad b_1 < 0 < b_2 \quad \text{and}\quad
 Q = \begin{pmatrix} 0 & Q_{12} \\ Q_{21} & 0 \end{pmatrix} \in \LL{2}.
\end{equation}

If $B = \begin{psmallmatrix} -1 & 0 \\ 0 & 1 \end{psmallmatrix}$
system~\eqref{eq:system} is equivalent to the Dirac system (see the classical
monographs~\cite{LevSar88},~\cite{Mar77}).

Let us associate linearly independent boundary conditions
\begin{equation} \label{eq:cond}
 U_j(y) := a_{j 1}y_1(0) + a_{j 2}y_2(0) + a_{j 3}y_1(1) + a_{j 4}y_2(1)= 0,
 \quad j \in \{1,2\},
\end{equation}
with system~\eqref{eq:system}, and denote as $L_U(Q) := L_U(Q)$ an operator, associated in $\fH := \LLV{2}$ with the boundary value problem (BVP)~\eqref{eq:system}--\eqref{eq:cond}. It is defined by differential expression $\cL$ on the domain
\begin{equation}
 \dom(L_U(Q)) = \{f \in {\rm AC}([0,1]; \bC^2) :\ \cL f \in \fH,
 \ U_1(f) = U_2(f) = 0\}.
\end{equation}

To the best of our knowledge, the spectral properties of the general $n \times n$ system of the form~\eqref{eq:system} with a nonsingular diagonal $n\times n$ matrix $B$ with complex entries and a potential matrix $Q(\cdot)$ of the form
\begin{equation} \label{1.2}
 B = \diag(b_1, \ldots, b_n) \in \bC^{n\times n} \quad \text{and}\quad
 Q(\cdot) =: (q_{jk}(\cdot))_{j,k=1}^n \in C^1([0,1]; \bC^{n\times n}).
\end{equation}
has first been investigated by G.D.~Birkhoff and R.E.~Langer~\cite{BirLan23}. Namely, they introduced the concepts of \emph{regular and strictly regular boundary conditions}~\eqref{eq:cond} and investigated the asymptotic behavior of eigenvalues and eigenfunctions of the corresponding operator $L_{U}(Q)$. Moreover, they proved \emph{a pointwise convergence result} on spectral decompositions of the operator $L_{U}(Q)$ corresponding to the BVP~\eqref{eq:system}--\eqref{eq:cond} with regular boundary conditions.

The completeness property in $L^2([0,1]; \bC^n)$ of the system of root vectors of BVP for general $n \times n$ system of the form~\eqref{eq:system} with matrices $B = \diag(b_1, \ldots, b_n)$ and $Q \in L^1([0,1]; \bC^{n\times n})$ was established for the first time by M.M.~Malamud and L.L.~Oridoriga in~\cite{MalOri12} for a wide class of BVPs, although for $2 \times 2$ Dirac system with $Q\in C([0,1]; \bC^{2 \times 2})$ it was proved earlier by V.A.~Marchenko in~\cite[Chapter 1.3]{Mar77}. As a development of~\cite{MalOri12},
in~\cite{AgiMalOri12,LunMal14IEOT,LunMal15JST,ALMO19}
completeness conditions for non-regular and even degenerate boundary conditions were found with applications to dissipative and normal operators.
In the joint paper~\cite{LunMal15JST} the author and M.M.~Malamud
also established the Riesz basis property with parentheses of the system of root vectors for different classes of BVPs for $n \times n$ system with arbitrary $B$ of the form~\eqref{1.2} and $Q \in L^{\infty}([0,1]; \bC^{n\times n})$.
Note also that BVP for $2m \times 2m$ Dirac equation ($B=\diag(-I_m, I_m)$) were investigated in~\cite{MykPuy13} (Bari-Markus property for Dirichlet BVP with $Q \in L^2([0,1]; \bC^{2m \times 2m})$ and in~\cite{KurAbd18,KurGad20} (Bessel and Riesz basis properties on abstract level).

The Riesz basis property in $\LLV{2}$ of BVP~\eqref{eq:system}--\eqref{eq:cond}
for $2 \times 2$ Dirac system ($b_2 = -b_1 = 1$) with various assumptions on the potential matrix $Q$ was investigated in numerous papers (see~\cite{TroYam01,TroYam02,HasOri09,DjaMit10,Bask11,DjaMit12UncDir,DjaMit12Crit,LunMal14Dokl,SavShk14,LunMal16JMAA}
and references therein). The case of separated boundary conditions and $Q \in C^1([0,1]; \bC^{2 \times 2})$ was treated in~\cite{TroYam01,TroYam02} and later in~\cite{HasOri09} for Dirac-type system. For Dirichlet and periodic boundary conditions the case $Q \in \LL{2}$ was first treated by P.~Djakov and B.~Mityagin~\cite{DjaMit10} and later by A.~Baskakov, A.~Derbushev, A.~Shcherbakov~\cite{Bask11}. Shortly after, P.~Djakov and B.~Mityagin~\cite{DjaMit12UncDir} extended these results to general regular boundary conditions.

The most complete result on the Riesz basis property for $2\times 2$ Dirac and Dirac-type systems, respectively, with $Q \in \LL{1}$ and strictly regular boundary conditions was obtained independently by different methods and at the same time by A.M.~Savchuk and A.A.~Shkalikov~\cite{SavShk14}
and by the author and M.M.~Malamud~\cite{LunMal14Dokl,LunMal16JMAA} (in~\cite{LunMal14Dokl} the sketches of the proof are given). The case of regular boundary conditions is treated in~\cite{SavShk14} for the first time. Other proofs were obtained later on in~\cite{SavSad15DAN,SavSad15,LunMal16JMAA} (see also recent survey~\cite{SavSad20} and references therein).

In~\cite{LunMal16JMAA} results for Dirac operator were also applied to Timoshenko beam model. In general one can show that dynamic generators of many physical models governed by systems of linear PDE are similar to certain first order differential operators. Hence spectral properties of such operators are of significant importance in the study of stability of solutions and corresponding $C_0$-semigroups of many physical models. In particular, in Section~\ref{sec:damped.string} we establish explicit connection between $2 \times 2$ Dirac type operators (with $b_1 \ne -b_2$) on the one hand and certain \textbf{non-canonical initial-boundary value problem for a damped string} on the other hand. This allows us to apply results on Bari and Riesz basis property (see Definition~\ref{def:bases}) for Dirac type operators obtained here and in~\cite{LunMal16JMAA} to the damped string equation.

Recall, that boundary conditions~\eqref{eq:cond} are called \textbf{regular}, if and only if they are \emph{equivalent} to the following conditions
\begin{equation} \label{eq:cond.canon.intro}
 \wh{U}_{1}(y) = y_1(0) + b y_2(0) + a y_1(1) = 0,\quad
 \wh{U}_{2}(y) = d y_2(0) + c y_1(1) + y_2(1) = 0,
\end{equation}
with certain $a,b,c,d \in \bC$ satisfying $ad-bc \ne 0$. Recall also that regular boundary conditions~\eqref{eq:cond} are called \textbf{strictly regular}, if the sequence $\{\l_n^0\}_{n \in \bZ}$ of the eigenvalues of the unperturbed ($Q=0$) BVP~\eqref{eq:system}--\eqref{eq:cond} (of the operator $L_U(0)$), is asymptotically separated. In particular, the eigenvalues $\{\l_n^0\}_{|n| > n_0}$ are geometrically and algebraically simple.

It is well known that \emph{non-degenerate separated} boundary conditions are always \emph{strictly regular}. Moreover, conditions~\eqref{eq:cond.canon.intro} \emph{are strictly regular for Dirac operator if and only if} $(a-d)^2 \ne -4bc$. In particular, antiperiodic ($a=d=0$, $b=c=1$) boundary conditions \emph{are regular but not strictly regular} for Dirac system, while they \emph{become strictly regular for Dirac-type system} if $-b_1, b_2 \in \bN$ and $b_2 - b_1$ is odd.

Note in this connection that \emph{periodic and antiperiodic (necessarily non-strictly regular) BVP} for $2 \times 2$ Dirac and Sturm-Liouville equations have also attracted certain attention during the last decade. For\ instance, a criterion for the system of root vectors of the \emph{periodic} BVP for $2 \times 2$ Dirac equation to contain a Riesz basis (without parentheses!) was obtained by P. Djakov and B. Mityagin in~\cite{DjaMit12Crit} (see also
recent papers~\cite{Mak19},~\cite{Mak20} by A.S.~Makin and the references therein). It is also worth mentioning that F.~Gesztesy and V.~Tkachenko~\cite{GesTka09,GesTka12} for $q \in L^2[0,\pi]$ and P.~Djakov and B.~Mityagin~\cite{DjaMit12Crit} for $q \in W^{-1,2}[0,\pi]$ established by different methods a \emph{criterion} for the system of root vectors to contain a Riesz basis for Sturm-Liouville operator $-\frac{d^2}{dx^2} + q(x)$ on $[0,\pi]$ (see also survey~\cite{Mak12}).

Let us emphasize that the proof of the Riesz basis property in~\cite{DjaMit10,Bask11,DjaMit12UncDir,MykPuy13} substantially relies on the Bari-Markus property: the quadratic closeness in $\LLV{2}$ of the spectral projectors of the operators $L_U(Q)$ and $L_U(0)$.
Assuming boundary conditions to be strictly regular, let $\{f_n\}_{n \in \bZ}$ and $\{f_n^0\}_{n \in \bZ}$ be the systems of root vectors of the operators $L_U(Q)$ and $L_U(0)$, respectively. Then Bari-Markus property states the implication: $Q \in L^2 \Rightarrow \sum_{n \in \bZ} \|f_n - f_n^0\|_2^2 < \infty$. Later, this property was generalized to the case $Q \in \LL{p}$, $p \in [1,2]$, in~\cite{SavShk14,Sad16,LunMal22JDE}. The most complete results in this direction were established
in the joint paper~\cite{LunMal22JDE} by the author and M.M.~Malamud. One of these results reads as follows.
\begin{theorem}[Theorem 7.15 in \cite{LunMal22JDE}] \label{th:ellp-close}
Let $\cK \in \LL{p}$ be a compact set for some $p \in [1,2]$, let $Q, \wt{Q} \in \cK$ and boundary conditions~\eqref{eq:cond} be strictly regular. Then for some normalized systems of root vectors $\{f_n\}_{n \in \bZ}$ and $\{\wt{f}_n\}_{n \in \bZ}$ of the operators $L_U(Q)$ and $L_U(\wt{Q})$ the following uniform relations hold for $Q, \wt{Q} \in \cK$:
\begin{align}
\label{eq:sum.fn-fn0}
 & \sum_{|n| > N} \|f_n - \wt{f}_n\|_{\infty}^{p'} \le C
 \|Q - \wt{Q}\|_p^{p'}, \qquad p \in (1,2], \quad 1/p'+1/p=1, \\
\label{eq:sum.fn-fn0.hardy}
 & \sum_{|n| > N} (1+|n|)^{p-2} \|f_n - \wt{f}_n\|_{\infty}^{p} \le
 C \|Q - \wt{Q}\|_p^p, \qquad p \in (1,2], \\
\label{eq:lim.fn-fn0.c0}
 & \lim_{n \to \infty} \sup_{Q, \wt{Q} \in \cK}
 \|f_n - \wt{f}_n\|_{\infty} = 0, \qquad p = 1.
\end{align}
\end{theorem}
Here and throughout the paper we denote by $\|f\|_s$ the $L^s$-norm of the element $f$ of a scalar, vector or matrix $L^s$-space.

Emphasize, that the proof of the estimates~\eqref{eq:sum.fn-fn0}--\eqref{eq:sum.fn-fn0.hardy} is based on the deep Carleson-Hunt theorem. Note, however, that these estimates with $\|\cdot\|_{p'}$-norm instead of $\|\cdot\|_{\infty}$-norm can be proved in a more direct way, which is elementary in character. Note also that these results substantially rely on transformation operators method that goes back to~\cite{Mal94,Mal99,LunMal16JMAA}.

Recall that the concepts of Riesz bases and bases quadratically close to the orthonormal bases were introduced by N.K.~Bari in~\cite{Bari51}. Results of this fundamental paper can also be found in the classical monograph~\cite{GohKre65} where a basis quadratically close to the orthonormal basis is called a Bari basis. Let us recall the definition of Riesz and Bari bases following~\cite[Section IV]{GohKre65}.
\begin{definition} \label{def:bases}
\textbf{(i)} A sequence of vectors $\{f_n\}_{n \in \bZ}$ in a separable Hilbert space $\fH$ is called a \textbf{Riesz basis} if it admits a representation $f_n = T e_n$, $n \in \bN$, where $\{e_n\}_{n \in \bZ}$ is an orthonormal basis in $\fH$ and $T : \fH \to \fH$ is a bounded operator with bounded inverse.

\textbf{(ii)} A sequence of vectors $\{f_n\}_{n \in \bZ}$ in a separable Hilbert space $\fH$ is called a \textbf{Bari basis} if it is quadratically close to some orthonormal basis $\{e_n\}_{n \in \bZ}$ in $\fH$, i.e.
\begin{equation} \label{eq:sum.fn-en}
 \sum_{n \in \bZ} \|f_n - e_n\|_{\fH}^2 < \infty.
\end{equation}
\end{definition}
A.S.~Markus in~\cite{Markus69} studied in detail bases of subspaces with the property similar to~\eqref{eq:sum.fn-en}. Bari basis property for different classes of differential operators was studied in~\cite{BDL00,Zhidkov02,Allah14}.
Note, however, that to the best of our knowledge the question of whether system of root vectors of the operator $L_U(Q)$ forms \emph{a Bari basis  has not been studied before}. Namely, results of papers~\cite{DjaMit10,Bask11,DjaMit12UncDir,MykPuy13,SavShk14,LunMal22JDE} in the case of $Q \in L^2$ and strictly regular boundary conditions establish quadratic closeness of systems of root vectors $\{f_n\}_{n \in \bZ}$ and $\{f_n^0\}_{n \in \bZ}$, but whether $\{f_n\}_{n \in \bZ}$ is quadratically close to some orthonormal basis $\{e_n\}_{n \in \bZ}$ remained an open question. The goal if this paper is to close this gap.
One of our main results establishes the criterion for the system of root vectors of the operator $L_U(Q)$ to form a Bari basis and reads as follows.
\begin{theorem} \label{th:crit.bari}
Let boundary conditions~\eqref{eq:cond.canon.intro} be strictly regular and let $Q \in \LL{2}$. Then some normalized system of root vectors of the operator $L_U(Q)$ is a Bari basis in $\LLV{2}$ if and only if the operator $L_U(0)$ is self-adjoint. The latter holds if and only if the coefficients $a,b,c,d$ in boundary conditions~\eqref{eq:cond.canon.intro} satisfy the following relations:
\begin{equation} \label{eq:abcd.sa.intro}
 |a|^2 + \beta |b|^2 = 1, \qquad
 |c|^2 + \beta |d|^2 = \beta, \qquad
 a \ol{c} + \beta b \ol{d} = 0, \qquad \beta := -b_2/b_1 > 0.
\end{equation}
In this case every normalized system of root vectors of the operator $L_U(Q)$ is a Bari basis in $\LLV{2}$.
\end{theorem}
Combining Theorem~\ref{th:crit.bari} with the results of the previous papers
\cite{DjaMit10,Bask11,DjaMit12UncDir,MykPuy13,LunMal14Dokl,LunMal16JMAA,SavShk14}
concerning the Riesz basis property we get the following surprising result.
\begin{corollary} \label{cor:not.bari}
Let $Q \in \LL{2}$ and let boundary conditions~\eqref{eq:cond.canon.intro} be strictly regular but not self-adjoint, i.e. the operator $L_U(0)$ is not self-adjoint. Then every normalized system of root vectors of the operator $L_U(Q)$ is \textbf{a Riesz basis but not a Bari basis} in $\LLV{2}$.
\end{corollary}
\section{Definitions and formulations of the main results}
Let us recall the following abstract criterion for Bari basis property.
\begin{proposition}~\cite[Theorem VI.3.2]{GohKre65} \label{prop:crit.bari}
A complete system $\fF = \{f_n\}_{n \in \bZ}$ of unit vectors in a separable Hilbert space $\fH$ forms a Bari basis if and only if there exists a sequence $\{g_n\}_{n \in \bZ}$ biorthogonal to $\fF$ that is quadratically close to $\fF$:
\begin{equation} \label{eq:sum.fn-gn.2}
\sum_{n \in \bZ} \|f_n - g_n\|_{\fH}^2 < \infty, \qquad (f_n, g_m)_{\fH} = \delta_{nm}, \quad n,m \in \bZ.
\end{equation}
\end{proposition}
Based on this abstract criterion we will introduce a generalization of Bari basis concept. Let $p \in [1,2]$ and $p' = p/(p-1) \in [2,\infty]$. It is well-known that for the dual space of $\ell^p := \ell^p(\bZ)$ we have,
\begin{equation} \label{eq:ellp*}
(\ell^p(\bZ))^* \cong \ell^{p'}(\bZ), \quad p \in (1,2],
\qquad\text{and}\qquad
(\ell^1(\bZ))^* \cong c_0(\bZ).
\end{equation}
For simplicity we identify $(\ell^p(\bZ))^*$ with $\ell^{p'}(\bZ)$ for $p \in (1,2]$ and with $c_0(\bZ)$ for $p=1$, respectively. E.g. $\{a_n\}_{n \in \bZ} \in (\ell^p(\bZ))^*$ for $p>1$ means that $\sum_{n \in \bZ} |a_n|^{p'} < \infty$.
With this in mind, we can extend Definition~\ref{def:bases}(ii) using equivalence from Proposition~\ref{prop:crit.bari} to more general concept of closeness of sequences $\{f_n\}_{n \in \bZ}$ and $\{g_n\}_{n \in \bZ}$.
\begin{definition} \label{def:bari.c0}
Let $p \in [1,2]$, let $\fF := \{f_n\}_{n \in \bZ}$ be a complete minimal sequence of unit vectors in a separable Hilbert space $\fH$ and let $\fG := \{g_n\}_{n \in \bZ}$ be its (unique) biorthogonal sequence: $(f_n, g_m)_{\fH} = \delta_{nm}$, $n, m \in \bZ$. A sequence $\fF$ is called a \textbf{Bari $(\ell^p)^*$-sequence} if it is ``($\ell^p)^*$-close'' to its biorthogonal sequence $\fG$, i.e. $\curl{\|f_n - g_n\|_{\fH}}_{n \in \bZ} \in (\ell^p)^*$. In view of~\eqref{eq:ellp*} it means that
\begin{equation} \label{eq:sum.fn-gn}
\sum_{n \in \bZ} \|f_n - g_n\|_{\fH}^{p'} < \infty \quad\text{if}\quad
p \in (1,2],
\quad\text{and}\quad
 \lim_{n \to \infty} \|f_n - g_n\|_{\fH} = 0
 \quad\text{if}\quad p = 1.
\end{equation}
For brevity we will call \emph{Bari $(\ell^1)^*$-sequence} as \textbf{Bari $c_0$-sequence} and \emph{Bari $(\ell^p)^*$-sequence} as \textbf{Bari $\ell^{p'}$-sequence} for $p \in (1,2]$.

\end{definition}
Proposition~\ref{prop:crit.bari} implies that the notion of Bari $\ell^2$-sequence coincides with the notion of Bari basis. Note also that every Bari $(\ell^p)^*$-sequence is Bari $c_0$-sequence. We specifically chose the word ``sequence'' because it is not clear if Bari $c_0$-sequence is a Riesz basis or even a regular basis in general case.
\begin{remark} \label{rem:c0.bari.diff}
Note that Bari $c_0$-property from definition~\ref{def:bari.c0} is not equivalent to more conventional formulation of $c_0$-closeness of $\{f_n\}_{n \in \bZ}$ to a certain orthonormal basis $\{e_n\}_{n \in \bZ}$ even if $\{f_n\}_{n \in \bZ}$ is already a Riesz basis. Indeed, in this case $f_n = e_n + K e_n$, where $K$ and $(I+K)^{-1}$ are bounded operators in $\fH$. Hence $\|f_n - e_n\|_{\fH} = \|K e_n\|_{\fH}$. It is easily seen that $g_n = \((I+K)^{-1}\)^* e_n = e_n - \((I+K)^{-1}\)^* K^* e_n$, and hence $\|g_n - e_n\| \to 0$ as $n \to \infty$ is equivalent to $\|K^* e_n\| \to 0$ as $n \to \infty$. If $K$ is not compact then $\lim_{n \to \infty}\|K e_n\| = 0$ is generally not equivalent to $\lim_{n \to \infty}\|K^* e_n\| = 0$ for a given orthonormal basis $\{e_n\}_{n \in \bZ}$.
\end{remark}
Let us also recall the notion of the system of root vectors of an operator with compact resolvent. Firt, we recall a few basic facts regarding the eigenvalues of a
compact, linear operator  $T \in \cB_{\infty}(\fH)$ in a
separable complex Hilbert space $\fH$. The {\it geometric
multiplicity}, $m_g(\lambda_0,T)$, of an eigenvalue $\lambda_0
\in \sigma_p (T)$ of $T$ is given by
$
m_g(\lambda_0,T) := \dim(\ker(T - \lambda_0)).
$

The {\it root subspace} of $T$ corresponding to $\lambda_0 \in
\sigma_p(T)$ is given by
\begin{equation}\label{root.subspace}
\cR_{\lambda_0}(T) = \big\{f\in\fH\,:\, (T - \lambda_0)^k f = 0 \
\ \text{for some}\ \ k\in\mathbb N \big\}.
\end{equation}
Elements of $\cR_{\l_0}(T)$ are called {\it root vectors}.
For $\lambda_0 \in \sigma_p (T) \backslash \{0\}$, the set
$\mathcal{R}_{\lambda_0}(T)$ is a closed linear subspace of
$\fH$ whose dimension equals to the {\it algebraic
multiplicity}, $m_a(\lambda_0,T)$, of $\lambda_0$,
$
m_a(\lambda_0,T) := \dim\big(\mathcal R_{\lambda_0}(T)\big)<\infty.
$

Denote by $\{\l_j\}_{j=1}^{\infty}$ the sequence of non-zero
eigenvalues of $T$ and let $n_j$ be the algebraic multiplicity
of $\l_j$. By the {\it system of root vectors} of the operator
$T$ we mean any sequence of the form
$
\cup_{j=1}^{\infty}\{e_{jk}\}_{k=1}^{n_j},
$
where $\{e_{jk}\}_{k=1}^{n_j}$ is a basis in $\cR_{\l_j}(T)$,
$n_j = m_a(\lambda_j,T) < \infty$. The system or root vectors of the operator $T$ is called \emph{normalized} if $\|e_{jk}\|_{\fH} = 1$, $j \in \bN$, $k \in \{1, \ldots, n_j\}$.

We are particularly interested in the case where $A$ is a
densely defined, closed, linear operator in $\fH$ whose
resolvent is compact, that is,
$
R_A(\l):=(A - \l)^{-1} \in \cB_{\infty}(\fH), \ \l \in \rho (A).
$
Via the spectral mapping theorem all eigenvalues of $A$
correspond to eigenvalues of its resolvent $R_A(\l)$, $\l \in
\rho (A)$, and vice versa. Hence, we use the same notions of
root vectors, root subspaces, geometric and algebraic
multiplicities associated with the eigenvalues of $A$, and the
system of root vectors of $A$.

Now we are ready to formulate the main result of this paper, which involve notions of Bari $(\ell^p)^*$-sequences and $c_0$-sequences from Definition~\ref{def:bari.c0} above.
\begin{theorem} \label{th:crit.lp.bari}
Let boundary conditions~\eqref{eq:cond.canon.intro} be strictly regular and let $Q \in \LL{p}$ for some $p \in [1,2]$. Then some normalized system of root vectors of the operator $L_U(Q)$ is a Bari $(\ell^p)^*$-sequence in $\LLV{2}$ if and only if the operator $L_U(0)$ is self-adjoint, i.e. when relations~\eqref{eq:abcd.sa.intro} hold for the coefficients $a,b,c,d$ in boundary conditions~\eqref{eq:cond.canon.intro}. In this case every normalized system of root vectors of the operator $L_U(Q)$ is a Bari $(\ell^p)^*$-sequence in $\LLV{2}$.
\end{theorem}
As an immediate consequence of Theorem~\ref{th:crit.lp.bari} we get Theorem~\ref{th:crit.bari}: the criterion of Bari basis property for Dirac-type operator $L_U(Q)$ with $L^2$-potential and strictly regular boundary conditions.

Let us briefly comment on the proof of our main result, Theorem~\ref{th:crit.lp.bari}.
First, we apply Theorem~\ref{th:ellp-close} to reduce the Bari $(\ell^p)^*$-property of the system of root vectors of operator $L_U(Q)$ with strictly regular boundary conditions to a certain explicit condition in terms of the eigenvalues $\{\l_n^0\}_{n \in \bZ}$ of the operator $L_U(0)$, which reads as follows for the case $p=1$.
\begin{proposition} \label{prop:c0.close.cond}
Let $Q \in \LL{1}$ and boundary conditions~\eqref{eq:cond} be strictly regular. Then some normalized systems of root vectors $\{f_n\}_{n \in \bZ}$ of the operator $L_U(Q)$ is a Bari $c_0$-sequence in $\LLV{2}$ if and only if:
\begin{equation} \label{eq:lim1.lim2.intro}
 b_1 |c| + b_2 |b| = 0, \qquad \lim_{n \to \infty} \Im \l_n^0 = 0
 \quad\text{and}\quad \lim_{n \to \infty} z_n = |bc|,
\end{equation}
where
\begin{equation} \label{eq:zn.def.intro}
 z_n := \(1 + d \exp(- i b_2 \l_n^0)\)\ol{\(1 + a \exp(i b_1 \l_n^0)\)},
\end{equation}
and $\{\l_n^0\}_{n \in \bZ}$ is the sequence of the eigenvalues of the operator $L_U(0)$, counting multiplicity.
\end{proposition}
With condition~\eqref{eq:lim1.lim2.intro} established, the main difficulty arises in reducing this condition to the desired explicit condition~\eqref{eq:abcd.sa.intro}.
In this connection, recall that the sequence $\{\l_n^0\}_{n \in \bZ}$ of the eigenvalues of the operator $L_U(0)$ coincides with the sequence of zeros of characteristic determinant
\begin{equation} \label{eq:Delta0.intro}
 \Delta_0(\l) = d + a e^{i (b_1+b_2) \l} + (ad-bc) e^{i b_1 \l}
 + e^{i b_2 \l}.
\end{equation}
If $b_2 / b_1 \in \bQ$ then the sequence $\{\l_n^0\}_{n \in \bZ}$ has a simple explicit form: it is the union of arithmetic progression that lie on the lines parallel to the real axis, which simplifies the problem a lot.

The case $b_2 / b_1 \notin \bQ$ is much more complicated.
Namely, if $|a|+|d|>0$ and $bc \ne 0$ there is no explicit description of the spectrum of the operator $L_U(0)$. Nevertheless, we were able to establish equivalence of~\eqref{eq:lim1.lim2.intro} and~\eqref{eq:abcd.sa.intro} using Weyl's equidistribution theorem (see~\cite[Theorem 4.2.2.1]{SteinShak03}). It implies the following crucial property of zeros of $\Delta_0(\cdot)$.
\begin{proposition} \label{prop:nlim.inf.intro}
Let $b_2/b_1 \notin \bQ$ and boundary conditions~\eqref{eq:cond.canon.intro} be regular, i.e. $ad-bc \ne 0$. Let $\{\l_n^0\}_{n \in \bZ}$ be the sequence of zeros of the characteristic determinant $\Delta_0(\cdot)$ counting multiplicity. Then each of the sequences $\{\exp(i b_1 \l_n^0)\}_{n \in \bZ}$ and $\{\exp(i b_2 \l_n^0)\}_{n \in \bZ}$ has infinite set of limit points.
\end{proposition}
This result was key for proving equivalence of~\eqref{eq:lim1.lim2.intro} and~\eqref{eq:abcd.sa.intro}, which in turn implies our main result, Theorem~\ref{th:crit.lp.bari}, and its main corollary, Theorem~\ref{th:crit.bari}.
\section{Regular and strictly regular boundary conditions}
\label{subsec:regular}
In this section we recall known properties of BVP~\eqref{eq:system}--\eqref{eq:cond} subject to regular or strictly regular boundary conditions from~\cite{LunMal16JMAA}.
Let us set
\begin{align}
 A := \begin{pmatrix} a_{11} & a_{12} & a_{13} & a_{14} \\
 a_{21} & a_{22} & a_{23} & a_{24} \end{pmatrix}, \qquad
\label{eq:Ajk.Jjk}
 A_{jk} := \begin{pmatrix} a_{1j} & a_{1k} \\ a_{2j} & a_{2k} \end{pmatrix},
 \quad J_{jk} := \det (A_{jk}), \quad j,k\in\{1,\ldots,4\}.
\end{align}
Let
\begin{equation} \label{eq:Phi.def}
 \Phi(\cdot, \l) =
 \begin{pmatrix} \varphi_{11}(\cdot, \l) & \varphi_{12}(\cdot, \l)\\
 \varphi_{21}(\cdot,\l) & \varphi_{22}(\cdot,\l)
 \end{pmatrix} =: \begin{pmatrix} \Phi_1(\cdot, \l) & \Phi_2(\cdot, \l)
 \end{pmatrix}, \qquad \Phi(0, \l) = I_2,
\end{equation}
be a fundamental matrix solution of the
system~\eqref{eq:system}, where $I_2 = \begin{psmallmatrix} 1 & 0 \\ 0 & 1 \end{psmallmatrix}$. Here $\Phi_k(\cdot, \l)$ is the
$k$th column of $\Phi(\cdot, \l)$.

The eigenvalues of the problem~\eqref{eq:system}--\eqref{eq:cond} counting multiplicity
are the zeros (counting multiplicity) of the characteristic determinant
\begin{equation} \label{eq:Delta.def}
 \Delta_Q(\l) := \det
 \begin{pmatrix}
 U_1(\Phi_1(\cdot,\l)) & U_1(\Phi_2(\cdot,\l)) \\
 U_2(\Phi_1(\cdot,\l)) & U_2(\Phi_2(\cdot,\l))
 \end{pmatrix}.
\end{equation}
Inserting~\eqref{eq:Phi.def} and~\eqref{eq:cond} into~\eqref{eq:Delta.def}, setting $\varphi_{jk}(\l) := \varphi_{jk}(1,\l)$, and taking notations~\eqref{eq:Ajk.Jjk} into account we arrive at the following expression for the characteristic determinant
\begin{equation} \label{eq:Delta}
 \Delta_Q(\l) = J_{12} + J_{34}e^{i(b_1+b_2)\l}
 + J_{32}\varphi_{11}(\l) + J_{13}\varphi_{12}(\l)
 + J_{42}\varphi_{21}(\l) + J_{14}\varphi_{22}(\l).
\end{equation}
If $Q=0$ we denote a fundamental matrix solution as $\Phi^0(\cdot, \l)$. Clearly
\begin{equation} \label{eq:Phi0.def}
 \Phi^0(x, \l)
 = \begin{pmatrix} e^{i b_1 x \l} & 0 \\ 0 & e^{i b_2 x \l} \end{pmatrix}
 =: \begin{pmatrix}
 \varphi_{11}^0(x, \l) & \varphi_{12}^0(x, \l)\\
 \varphi_{21}^0(x,\l) & \varphi_{22}^0(x,\l)
 \end{pmatrix}
 =: \begin{pmatrix} \Phi_1^0(x, \l) & \Phi_2^0(x, \l) \end{pmatrix},
\end{equation}
for $x \in [0,1]$ and $\l \in \bC$. Here $\Phi_k^0(\cdot, \l)$ is the $k$th column of $\Phi^0(\cdot, \l)$. In
particular, the characteristic determinant $\Delta_0(\cdot)$ becomes
\begin{equation} \label{eq:Delta0}
 \Delta_0(\l) = J_{12} + J_{34}e^{i(b_1+b_2)\l}
 + J_{32}e^{ib_1\l} + J_{14}e^{ib_2\l}.
\end{equation}
In the case of Dirac system $(B =\diag (-1,1))$ this formula is
simplified to
\begin{equation} \label{eq:Delta0_Dirac}
 \Delta_0(\l) = J_{12} + J_{34} + J_{32}e^{-i\l} + J_{14}e^{i\l}.
\end{equation}
Let us recall the definition of regular boundary conditions.
\begin{definition} \label{def:regular}
Boundary conditions~\eqref{eq:cond} are called \textbf{regular} if
\begin{equation} \label{eq:J32J14ne0}
 J_{14} J_{32} \ne 0.
\end{equation}
\end{definition}
Let us recall one more definition (cf.~\cite{Katsn71}).
\begin{definition} \label{def:incompressible}
Let $\L := \{\l_n\}_{n \in \bZ}$ be a sequence of complex numbers. It is
called \textbf{incompressible} if for some $d \in \bN$ every rectangle
$[t-1,t+1] \times \bR \subset \bC$ contains at most $d$ entries of the sequence,
i.e.
\begin{equation} \label{eq:card.incomp}
 \card\{n \in \bZ : |\Re \l_n - t| \le 1 \} \le d, \quad t \in \bR.
\end{equation}
\end{definition}
Recall that $\bD_r(z) \subset \bC$ denotes the disc of radius $r$ with a
center $z$.

Let us recall certain important properties from~\cite{LunMal16JMAA} of the characteristic determinant $\Delta(\cdot)$ in the case of regular boundary conditions.
\begin{proposition}~\cite[Proposition 4.6]{LunMal16JMAA} \label{prop:sine.type}
Let the boundary conditions~\eqref{eq:cond} be regular. Then the characteristic determinant $\Delta_Q(\cdot)$ of the problem~\eqref{eq:system}--\eqref{eq:cond} given by~\eqref{eq:Delta}
has infinitely many zeros $\L := \{\l_n\}_{n \in \bZ}$ counting multiplicities and
\begin{equation} \label{eq:ln.in.Pih}
|\Im \l_n| \le h, \quad n \in \bZ, \qquad\text{for some}\ \ h \ge 0.
\end{equation}
Moreover, the sequence $\L$ is incompressible
and can be ordered in such a way that the following asymptotical formula holds
\begin{equation} \label{eq:lam.n=an+o1}
 \Re \l_n = \frac{2 \pi n}{b_2 - b_1} (1 + o(1)) \quad\text{as}\quad n \to\infty.
\end{equation}
\end{proposition}
Clearly, the conclusions of Proposition~\ref{prop:sine.type} are valid for the characteristic determinant $\Delta_0(\cdot)$ given by~\eqref{eq:Delta0}. Let $\L_0 = \{\l_n^0\}_{n \in \bZ}$ be the sequence of its zeros counting multiplicity. Let us order the sequence $\L_0$ in a (possibly non-unique) way such that $\Re \l_n^0 \le \Re \l_{n+1}^0$, $n \in \bZ$.
Let us recall an important result from~\cite{LunMal14Dokl,LunMal16JMAA}
and~\cite{SavShk14} concerning asymptotic behavior of the eigenvalues.
\begin{proposition}[Proposition 4.7 in~\cite{LunMal16JMAA}]
\label{prop:Delta.regular.basic}
Let $Q \in \LL{1}$ and let boundary conditions~\eqref{eq:cond} be regular. Then the sequence $\L = \{\l_n\}_{n \in \bZ}$ of zeros of $\Delta_Q(\cdot)$ can be ordered in such a way that the following asymptotic formula holds
\begin{equation} \label{eq:l.n=l.n0+o(1)}
 \l_n = \l_n^0 + o(1), \quad\text{as}\quad n \to \infty, \quad n \in \bZ.
\end{equation}
\end{proposition}
To define strictly regular boundary conditions we need the following definition.
\begin{definition} \label{def:sequences}
\textbf{(i)} A sequence $\L := \{\l_n\}_{n \in \bZ}$ of complex numbers is
said to be \textbf{separated} if for some positive $\tau > 0,$
\begin{equation} \label{separ_cond}
 |\l_j - \l_k| > 2 \tau \quad \text{whenever}\quad j \ne k.
\end{equation}
In particular, all entries of a separated sequence are distinct.

\textbf{(ii)} The sequence $\L$ is said to be \textbf{asymptotically
separated} if for some $N \in \bN$ the subsequence $\{\l_n\}_{|n| > N}$ is
separated.
\end{definition}
Let us recall a notion of strictly regular boundary conditions.
\begin{definition} \label{def:strictly.regular}
Boundary conditions~\eqref{eq:cond} are called \textbf{strictly regular}, if they
are regular, i.e. $J_{14} J_{32} \ne 0$, and the sequence of zeros $\l_0 =
\{\l_n^0\}_{n \in \bZ}$ of the characteristic determinant $\Delta_0(\cdot)$ is
asymptotically separated. In particular, there exists $n_0$ such that zeros
$\{\l_n^0\}_{|n| > n_0}$ are geometrically and algebraically simple.
\end{definition}
It follows from Proposition~\ref{prop:Delta.regular.basic} that the sequence $\L = \{\l_n\}_{n \in \bZ}$ of zeros of $\Delta_Q(\cdot)$ is asymptotically separated if the boundary conditions are strictly regular.

Assuming boundary conditions~\eqref{eq:cond} to be regular, let us rewrite them in
a more convenient form. Since $J_{14} \ne 0$, the inverse matrix $A_{14}^{-1}$
exists. Therefore writing down boundary conditions~\eqref{eq:cond} as the vector
equation $\binom{U_1(y)}{U_2(y)} = 0$ and multiplying it by the matrix
$A_{14}^{-1}$ we transform these conditions as follows
\begin{equation} \label{eq:cond.canon}
\begin{cases}
 \wh{U}_{1}(y) = y_1(0) + b y_2(0) + a y_1(1) = 0, \\
 \wh{U}_{2}(y) = d y_2(0) + c y_1(1) + y_2(1) = 0,
\end{cases}
\end{equation}
with some $a,b,c,d \in \bC$. Now $J_{14} = 1$ and the boundary conditions
~\eqref{eq:cond.canon} are regular if and only if $J_{32} = ad-bc \ne 0$. Thus, the
characteristic determinants $\Delta_0(\cdot)$ and $\Delta(\cdot)$ take the form
\begin{align}
\label{eq:Delta0.new}
 \Delta_0(\l) &= d + a e^{i (b_1+b_2) \l} + (ad-bc) e^{i b_1 \l}
 + e^{i b_2 \l}, \\
\label{eq:Delta.new}
 \Delta(\l) &= d + a e^{i (b_1+b_2) \l} + (ad-bc) \varphi_{11}(\l)
 + \varphi_{22}(\l) + c \varphi_{12}(\l) + b \varphi_{21}(\l).
\end{align}
\begin{remark} \label{rem:cond.examples}
Let us list some types of \emph{strictly regular} boundary
conditions~\eqref{eq:cond.canon}. In all of these cases except 4b the set of zeros
of $\Delta_0$ is a union of finite number of arithmetic progressions.

\begin{enumerate}

\item Regular boundary conditions~\eqref{eq:cond.canon} for Dirac operator ($-b_1 = b_2 = 1$) are
strictly regular if and only if $(a-d)^2 \ne -4bc$.

\item Separated boundary conditions ($a=d=0$, $bc \ne 0$) are always strictly regular.

\item Let $b_2 / b_1 \in \bQ$, i.e. $b_1 = -n_1 b_0$, $b_2 = n_2 b_0$, $n_1, n_2 \in \bN$, $b_0 > 0$ and $\gcd(n_1,n_2)=1$. Since $ad \ne bc$, $\Delta_0(\cdot) e^{-i b_1 \l}$ is a polynomial in $e^{i b_0 \l}$ of degree $n_1 + n_2$ with non-zero roots. Hence, boundary conditions~\eqref{eq:cond.canon} are strictly regular if and only if this polynomial does not have multiple roots. Let us list some cases with explicit conditions.

\begin{enumerate}

\item~\cite[Lemma 5.3]{LunMal16JMAA} Let $ad \ne 0$ and $bc=0$. Then
boundary conditions~\eqref{eq:cond.canon} are strictly regular if and only if
\begin{equation} \label{eq:bc=0.crit.rat}
 b_1 \ln |d| + b_2 \ln |a| \ne 0 \quad\text{or}\quad
 n_1 \arg(-d) - n_2 \arg(-a) \notin 2 \pi \bZ.
\end{equation}

\item In particular, antiperiodic boundary conditions ($a=d=1$, $b=c=0$) are strictly regular if
and only if $n_1 - n_2$ is odd. Note that these boundary conditions are not strictly regular in
the case of a Dirac system.

\item~\cite[Proposition 5.6]{LunMal16JMAA} Let $a=0$, $bc \ne 0$. Then
boundary conditions~\eqref{eq:cond.canon} are strictly regular if and only if
\begin{equation} \label{eq:a=0.crit.rat}
 n_1^{n_1} n_2^{n_2} (-d)^{n_1 + n_2} \ne (n_1 + n_2)^{n_1 + n_2} (-b c)^{n_2}.
\end{equation}

\end{enumerate}

\item Let $\alp := -b_1 / b_2 \notin \bQ$. Then the problem of strict regularity of boundary conditions is generally much more complicated. Let us list some known cases:

\begin{enumerate}

\item~\cite[Lemma 5.3]{LunMal16JMAA} Let $ad \ne 0$ and $bc=0$. Then
boundary conditions~\eqref{eq:cond.canon} are strictly regular if and only if
\begin{equation} \label{eq:bc=0.crit.irrat}
 b_1 \ln |d| + b_2 \ln |a| \ne 0.
\end{equation}

\item~\cite[Proposition 5.6]{LunMal16JMAA} Let $a=0$ and $bc, d \in \bR
\setminus \{0\}$. Then boundary conditions~\eqref{eq:cond.canon} are strictly regular if and only if
\begin{equation} \label{eq:a=0.crit}
 d \ne -(\alp+1)\(|bc| \alp^{-\alp}\)^{\frac{1}{\alp+1}}.
\end{equation}

\end{enumerate}

\end{enumerate}
\end{remark}
It is well-known that the biorthogonal system to the system of root vectors of the operator $L_U(Q)$ coincides with the system of root vectors of the adjoint operator $L_U^*(Q) := (L_U(Q))^*$ after proper normalization. In this connection we give the explicit form of the operator $L_U(Q)^*$ in the case of boundary conditions~\eqref{eq:cond.canon}.
\begin{lemma}
\label{lem:adjoint}
Let $L_{U}(Q)$ be an operator corresponding to the problem~\eqref{eq:system}, \eqref{eq:cond.canon}. Then the adjoint operator $L_U^*(Q)$ is given by the differential expression~\eqref{eq:system} with $Q^*(x) = \begin{pmatrix} 0 & \ol{Q_{21}(x)} \\ \ol{Q_{12}(x)} & 0 \end{pmatrix}$ instead of $Q$ and the boundary conditions
\begin{equation} \label{eq:cond*}
\begin{cases}
 U_{*1}(y) = \ol{a} y_1(0) + y_1(1) + \beta^{-1} \ol{c} y_2(1) &= 0, \\
 U_{*2}(y) = \beta \ol{b} y_1(0) + y_2(0) + \ol{d} y_2(1) &= 0,
\end{cases}
\end{equation}
where as before $\beta = - b_2/b_1 > 0$. I.e. $L_U^*(Q) = L_{U*}(Q^*)$. Moreover, boundary conditions~\eqref{eq:cond*} are regular (strictly regular) simultaneously with boundary conditions~\eqref{eq:cond.canon}.
\end{lemma}
\begin{corollary} \label{cor:sa.crit}
The operator $L_U(0)$ corresponding to the problem~\eqref{eq:system},~\eqref{eq:cond.canon} with $Q=0$ is selfadjoint if and only if
\begin{equation} \label{eq:abcd.sa2}
a = \ol{d} u, \quad d = \ol{a} u, \quad b = -\beta^{-1} \ol{c} u,
\quad c = -\beta \ol{b} u, \qquad u := ad-bc \ne 0,
\end{equation}
which in turn is equivalent to~\eqref{eq:abcd.sa.intro}.
\end{corollary}
\begin{proof}
Boundary conditions~\eqref{eq:cond.canon} and~\eqref{eq:cond*} can be rewriten in a matrix form as
\begin{equation}
\binom{y_1(0)}{y_2(1)} + \begin{pmatrix} a & b \\ c & d \end{pmatrix}
\binom{y_1(1)}{y_2(0)} = 0 \quad \text{and} \quad
\begin{pmatrix} \ol{a} & \beta^{-1} \ol{c} \\ \beta \ol{b} & \ol{d} \end{pmatrix} \binom{y_1(0)}{y_2(1)} +
\binom{y_1(1)}{y_2(0)} = 0,
\end{equation}
respectively. Hence boundary conditions~\eqref{eq:cond.canon} and~\eqref{eq:cond*} are equivalent if and only if
\begin{equation}
\begin{pmatrix} \ol{a} & \beta^{-1} \ol{c} \\ \beta \ol{b} & \ol{d} \end{pmatrix}  = \begin{pmatrix} a & b \\ c & d \end{pmatrix}^{-1} =
\frac{1}{ad-bc}\begin{pmatrix} d & -b \\ -c & a \end{pmatrix} =
\frac{1}{u}\begin{pmatrix} d & -b \\ -c & a \end{pmatrix},
\end{equation}
which is equivalent to~\eqref{eq:abcd.sa2}.

On the other hand we can rewrite condtions~\eqref{eq:cond.canon} as
\begin{equation} \label{eq:Cy0+Dy1}
C y(0) + D y(1) = 0, \qquad
C = \begin{pmatrix} 1 & b \\ 0 & d \end{pmatrix}, \qquad
D = \begin{pmatrix} a & 0 \\ c & 1 \end{pmatrix}.
\end{equation}
According to~\cite[Lemma 5.1]{LunMal14IEOT}
operator $L_U(0)$ with boundary conditions rewritten as~\eqref{eq:Cy0+Dy1} is selfadjoint if and only if $C B C^* = D B D^*$. Straightforward calculations show that
\begin{align}
\label{eq:CBC*}
 b_1^{-1} C B C^* = b_1^{-1} \begin{pmatrix} b_1 + b_2 |b|^2 &
 b_2 b  \ol{d} \\ b_2 \ol{b} d & b_2 |d|^2 \end{pmatrix} &=
 \begin{pmatrix} 1 - \beta |b|^2 & -\beta b \ol{d} \\ -\beta \ol{b} d &
 -\beta |d|^2 \end{pmatrix}, \\
\label{eq:DBD*}
 b_1^{-1} D B D^* = b_1^{-1} \begin{pmatrix} b_1 |a|^2 & b_1 a \ol{c} \\
 b_1 \ol{a} c & b_1 |c|^2 + b_2 \end{pmatrix} &= \begin{pmatrix}
 |a|^2 & a \ol{c} \\ \ol{a} c & |c|^2 - \beta \end{pmatrix}.
\end{align}
Hence $C B C^* = D B D^*$ is equivalent
to the condition~\eqref{eq:abcd.sa.intro}. It is interesting to note that establishing equivalence of~\eqref{eq:abcd.sa.intro} and~\eqref{eq:abcd.sa2} directly is somewhat tedious.
\end{proof}
\section{Properties of the spectrum of the unperturbed operator} \label{sec:unperturb}
In this section we obtain some properties of the sequence $\{\l_n^0\}_{n \in \bZ}$ of the characteristic determinant $\Delta_0(\cdot)$ in the case of regular boundary conditions~\eqref{eq:cond.canon} that will be needed in Section~\ref{sec:bari.c0} to study Bari $c_0$-property of the system of root vectors of the operator $L_U(0)$ (see Definition~\ref{def:bari.c0}).
Recall that $x_n \asymp y_n$, $n \in \bZ$, means that there exists $C_2 > C_1 > 0$ such that $C_1 |y_n| \le |x_n| \le C_2 |y_n|$, $n \in \bZ$. We start the following simple property of zeros of $\Delta_0(\cdot)$.
\begin{lemma} \label{lem:ln0.exp.asymp}
Let boundary conditions~\eqref{eq:cond.canon} be regular and $\L_0 := \{\l_n^0\}_{n \in \bZ}$ be the sequence of zeros of $\Delta_0(\cdot)$ counting multiplicity. Set
\begin{equation} \label{eq:ekn.def}
 e_{1n} := e_{1,n} := e^{i b_1 \l_n^0}, \qquad
 e_{2n} := e_{2,n} := e^{-i b_2 \l_n^0},
 \qquad n \in \bZ.
\end{equation}

\textbf{(i)} Let $bc \ne 0$. Then
\begin{equation} \label{eq:1+ae1.1+de2}
 1 + a e_{1n} \asymp 1,
 \qquad 1 + d e_{2n} \asymp 1, \qquad n \in \bZ.
\end{equation}

\textbf{(ii)} Let boundary conditions~\eqref{eq:cond.canon} be strictly regular. Then
\begin{equation} \label{eq:|1+de|+|1+ae|.asymp.1}
 \abs{1 + a e_{1n}}^2 + \abs{1 + d e_{2n}}^2
 \asymp 1, \quad n \in \bZ.
\end{equation}
\end{lemma}
\begin{proof}
Note that
\begin{equation}
 \Delta_0(\l)
 = \(1 + a e^{i b_1 \l}\) \(d + e^{i b_2 \l}\) - b c \cdot e^{i b_1 \l}
 = e^{i b_2\l} \(1 + a e^{i b_1 \l}\) \(1 + d e^{-i b_2\l}\)
 - b c \cdot e^{i b_1 \l}, \quad \l \in \bC.
\end{equation}
Since $\Delta(\l_n^0) = 0$, $n \in \bZ$, then with account of notation~\eqref{eq:ekn.def} we have
\begin{equation} \label{eq:Delta_0_in_roots}
 \(1 + a e_{1n}\) \(1 + d e_{2n}\)
 = b c e_{1n} e_{2n}, \qquad n \in \bZ,
\end{equation}
According to Proposition~\ref{prop:sine.type} the relation~\eqref{eq:ln.in.Pih} holds.
Hence
\begin{equation} \label{eq:e.bj.ln.asymp.1}
 e_{jn} \asymp 1, \qquad n \in \bZ, \quad j \in \{1, 2\}.
\end{equation}

\textbf{(i)} Since $bc \ne 0$, then combining~\eqref{eq:Delta_0_in_roots} with~\eqref{eq:e.bj.ln.asymp.1} yields the following estimate with some $C_3 > C_2 > C_1 > 0$,
\begin{equation} \label{eq:||+||>=bc.e}
 C_3 > C_2 \abs{1 + a e_{1n}}
 \ge \abs{(1 + a e_{1n})(1 + d e_{2n})}
 = 2 |bc| \cdot \abs{e_{1n} e_{2n}} > C_1, \quad |n| \in \bZ,
\end{equation}
which proves the first relation in~\eqref{eq:1+ae1.1+de2}. The second relation is proved similarly.

\textbf{(ii)} If $bc \ne 0$ then~\eqref{eq:|1+de|+|1+ae|.asymp.1} is implied by~\eqref{eq:1+ae1.1+de2}. Let $b c = 0$. In this case $a d \ne 0$ and $\Delta_0(\l) = e^{i b_2\l} \(1 + a e^{i b_1 \l}\) \(1 + d e^{-i b_2\l}\)$. It is clear that $\L_0 = \L_0^1 \cup \L_0^2$, where $\L_0^1 = \{\l_{1,n}^{0}\}_{n \in \bZ}$ and $\L_0^2 = \{\l_{2,n}^{0}\}_{n \in \bZ}$ are the sequences of zeros of the first and second factor, respectively. Clearly, these sequences constitute arithmetic progressions lying on the lines, parallel to the real axis. More precisely,
\begin{equation} \label{eq:l1n.l2n.bc=0}
 \l_{1,n}^{0} =
 \frac{\arg(-a^{-1}) + 2 \pi n}{b_1} + i\frac{\ln|a|}{b_1},
 \qquad
 \l_{2,n}^{0} = \frac{\arg(-d) + 2 \pi n}{b_2} - i\frac{\ln|d|}{b_2},
\end{equation}
for $n \in \bZ$. Since boundary conditions~\eqref{eq:cond.canon} are strictly regular, then the union of these arithmetic progressions $\L_0 = \L_0^1 \cup \L_0^2$ is asymptotically separated. It is easily seen that, in fact, $\L_0$ is separated: if $b_2/b_1 \in \bQ$ then $\L_0$ is periodic and if $b_2/b_1 \notin \bQ$ then arithmetic progressions $\L_0^1$ and $\L_0^2$ necessarily lie on different parallel lines.
This implies the following asymptotic relations:
\begin{equation} \label{eq:1+de.1+ae.asymp.1}
 1 + d e^{-i b_2 \l_{2,n}^0} \asymp 1, \quad
 1 + a e^{i b_1 \l_{1,n}^0} \asymp 1, \qquad n \in \bZ,
\end{equation}
Since $\L_0 = \L_0^1 \cup \L_0^2$, relations~\eqref{eq:1+de.1+ae.asymp.1} trivially imply~\eqref{eq:|1+de|+|1+ae|.asymp.1}.
\end{proof}
Throughout the rest of the section we will denote by $\fr{x} := x - \floor{x}$ the fractional part of $x \in \bR$. To treat the tricky case of $\beta = -b_2/b_1 \notin \bQ$, we need Weyl's equidistribution theorem (see~\cite[Theorem 4.2.2.1]{SteinShak03}). More precisely, we need the following its consequence.
\begin{lemma} \label{lem:weyl}
Let $\beta \in \bR \setminus \bQ$ and $0 \le a < b \le 1$. Then for any $\eps>0$ there exists $M_{a,b,\eps} > 0$ such that for $M \in \bN$ we have:
\begin{equation}
 \card\{m \in \{-M, \ldots, M\} : \fr{\beta m} \in [a,b]\}
 \le 2 (b - a + \eps) M, \qquad M \ge M_{a,b,\eps}.
\end{equation}
\end{lemma}

First, let us recall some simple properties of the sequences that have a finite set of limit points. For brevity we denote the cardinality of the limit points set of a bounded sequence $\{z_n\}_{n \in \bZ} \subset \bC$ as $\nlim\{z_n\}_{n \in \bZ}$,
\begin{multline}
 \nlim\curl{z_n}_{n \in \bZ} := \card\left\{z \in \bC :
 \lim_{k \to \infty} z_{n_k} = z \right. \\
 \left. \text{for some} \ \ \{n_k\}_{k \in \bN} \subset \bZ \ \ \text{such that} \ \ n_j \ne n_k
 \ \ \text{for} \ \ j \ne k \right\}.
\end{multline}
If the set of limit points is infinite we set $\nlim\curl{z_n}_{n \in \bZ} := \infty$.
\begin{lemma} \label{lem:limit}
The following statements hold:

\begin{enumerate}
\item[(i)] Let $\{a_n\}_{n \in \bZ} \subset \bC$ be bounded, $f$ be continuous on $\cup_{|n| > N} \ol{\bD_{\eps}(a_n)}$ for some $\eps>0$ and $N > 0$, and $\nlim \{a_n\}_{n \in \bZ} = m \in \bN$. Then $\nlim \{f(a_n)\}_{n \in \bZ} \le m$.
\item[(ii)] Let $\{a_n\}_{n \in \bZ} \subset \bC$ and $\{b_n\}_{n \in \bZ} \subset \bC$ be bounded sequences and let $\nlim \{a_n\}_{n \in \bZ} = m_a$ and $\nlim \{b_n\}_{n \in \bZ} = m_b \in \bN$. Then $\nlim \{a_n + b_n\}_{n \in \bZ} \le m_a m_b$ and $\nlim \{a_n b_n\}_{n \in \bZ} \le m_a m_b$.
\item[(iii)] Let $y_n \in [0, 1)$, $n \in \bZ$, and let $\nlim \{\sin (2 \pi y_n)\}_{n \in \bZ} = m \in \bN$. Then $\nlim \{y_n\}_{n \in \bZ} \le 2m+1$.
\item[(iv)] Let $a, b \in \bR$, $\{x_n\}_{n \in \bZ} \subset \bR$ be bounded and $\nlim \{x_n\}_{n \in \bZ} = m \in \bN$. Then $$\nlim\{\fr{a x_n + b}\}_{n \in \bZ} \le m+1.$$
\end{enumerate}
\end{lemma}
The following result of Diophantine approximation nature plays crucial role in treating the tricky case of $b_2/b_1 \notin \bQ$.
\begin{lemma} \label{lem:sin.weyl}
Let $b_1,b_2 \in \bR \setminus \{0\}$ and $b_2/b_1 \notin \bQ$. Further, let $\{\alp_n\}_{n \in \bZ} \subset \bR$ be an incompressible sequence such that
\begin{equation} \label{eq:card.alpn>}
\card\{n \in \bZ : |\alp_n| \le M\} \ge \gamma M,
\qquad M \ge M_0,
\end{equation}
for some $\gamma, M_0 > 0$.
Then one of the sequences $\{\sin(b_1 \alp_n)\}_{n \in \bZ}$ and $\{\sin(b_2 \alp_n)\}_{n \in \bZ}$
has an infinite set of limit points.
\end{lemma}
\begin{proof}
Assume the contrary. Namely, let
$$
\nlim \curl{\sin(b_1 \alp_n)}_{n \in \bZ} = m_1 \in \bN
\qquad\text{and}\qquad
\nlim \curl{\sin(b_2 \alp_n)}_{n \in \bZ} = m_2 \in \bN.
$$
Let us set
\begin{equation} \label{eq:psi1}
 b_1 \alp_n = 2 \pi (k_n + \delta_n), \qquad
 k_n := \floor{\frac{b_1 \alp_n}{2 \pi}} \in \bZ, \quad
 \delta_n = \fr{\frac{b_1 \alp_n}{2 \pi}} \in [0, 1).
\end{equation}
It is clear that $\sin (2 \pi \delta_n) = \sin(b_1 \alp_n)$. Hence by Lemma~\ref{lem:limit}(iii)
\begin{equation} \label{eq:nlim.deltan}
 \nlim \{\delta_n\}_{n \in \bZ} \le 2m_1+1.
\end{equation}
It is clear from~\eqref{eq:psi1} that
$$
b_2 \alp_n = 2 \pi \( \beta k_n + \beta \delta_n\),
\quad n \in \bZ, \qquad \beta := b_2/b_1 \notin \bQ.
$$
The same reasoning as above shows that
$$
\nlim \curl{u_n}_{n \in \bZ} \le 2m_2+1, \qquad
u_n := \fr{\beta k_n + \beta \delta_n}.
$$
Further, combining~\eqref{eq:nlim.deltan} with by Lemma~\ref{lem:limit}(iv) implies that
\begin{equation}
\nlim\curl{v_n}_{n \in \bZ} \le 2m_1+2, \quad\text{where}\quad
v_n := \fr{\beta \delta_n}, \quad n \in \bZ.
\end{equation}
Finally, note that $\fr{\beta k_n} = \fr{u_n - v_n}$, $n \in \bZ$. Hence by parts (ii) and (iv) of Lemma~\ref{lem:limit}
the sequence $\curl{\fr{\beta k_n}}_{n \in \bZ}$ has exactly $p \le (2m_2+1)(2m_1+2)+1$ limit points $0 \le x_1 < \ldots < x_p \le 1$.

Let $\eps > 0$ be fixed. Then there exists $N_{\eps} \in \bN$ such that
\begin{equation} \label{eq:beta.kn.in.Ieps}
 \fr{\beta k_n} \in \cI_{\eps} := [0,1) \cap
 \bigcup_{j=1}^p (x_j-\eps, x_j+\eps), \qquad |n| \ge N_{\eps}.
\end{equation}
Since $\beta \notin \bQ$, Lemma~\ref{lem:weyl} implies that
\begin{align} \label{eq:card.JepsM}
 \card(\cJ_{\eps,M}) & \le 6 p \eps M,
 \qquad M \ge M_{\eps}, \quad M \in \bN, \quad\text{where} \\
\label{eq:JepsM.def}
 \cJ_{\eps,M} & :=
 \{m \in \{-M, \ldots, M\} : \fr{\beta m} \in \cI_{\eps}\},
 \quad M \in \bN,
\end{align}
For $M_{\eps} := \max\bigl\{M_{x_j-\eps,x_j+\eps,\eps} : j \in \{1,\ldots, p\}\bigr\}$.

Let $M \in \bN$ and consider the set
$$
\cK_{\eps,M} := \curl{|n| \ge N_{\eps} : |k_n| \le M} \subset \bZ,
$$
It is clear from~\eqref{eq:psi1} and inequality $|[x]| < |x|+1$ that
$$
\cK_{\eps,M} \supset \curl{|n| \ge N_{\eps} : |\alp_n| \le \wt{M}},
\qquad \wt{M} := \frac{2 \pi (M-1)}{|b_1|}.
$$
Hence if $\wt{M} \ge M_0$ condition~\eqref{eq:card.alpn>} implies that
\begin{equation} \label{eq:card.KepsM>}
 \card(\cK_{\eps,M}) \ge \gamma \wt{M} - 2 N_{\eps} + 1 \ge \gamma_1 M, \qquad M \ge \wt{M}_{\eps},
\end{equation}
with $\gamma_1 := \pi \gamma |b_1^{-1}| > 0$ and some $\wt{M}_{\eps} \ge M_{\eps}$.
Condition~\eqref{eq:beta.kn.in.Ieps} and definition~\eqref{eq:JepsM.def} of $\cJ_{\eps,M}$ imply that for $n \in \cK_{\eps,M}$ we have $k_n \in \cJ_{\eps,M}$. Since $\curl{\alp_n}_{n \in \bZ}$ is incompressible then so is $\curl{k_n}_{n \in \bZ}$. Hence multiplicities $d_m := \card\curl{n \in \bZ : k_n = m}$ are bounded, $d_m \le d$, $m \in \bZ$, for some $d \in \bN$. Hence for every $m \in \cJ_{\eps,M}$ there are at most $d$ values of $n \in \cK_{\eps,M}$ for which $k_n = m$. Combining this observation with the estimate~\eqref{eq:card.JepsM} we arrive at
\begin{equation} \label{eq:card.KepsM<}
 \card\(\cK_{\eps,M}\) \le d \card\(\cJ_{\eps,M}\) \le 6 d p \eps M.
\end{equation}
Now picking $\eps > 0$ such that that $6 d p \eps < \gamma_1$ and $M > \wt{M}_{\eps}$ we see that cardinality estimates~\eqref{eq:card.KepsM>} and~\eqref{eq:card.KepsM<} contradict to each other, which finishes the proof.
\end{proof}
\begin{remark}
It is clear from the proof of Lemma~\ref{lem:sin.weyl} that the statement remains valid if we relax condition~\eqref{eq:card.alpn>} to only hold for $M \in \cM \subset \bN$, where $\cM$ is some fixed unbounded subset of $\bN$.
\end{remark}
To apply Lemma~\ref{lem:sin.weyl} we first need to establish property~\ref{lem:sin.weyl} for the sequence $\{\Re \l_n^0\}_{n \in \bZ}$. It easily follows from the asymptotic formula~\eqref{eq:lam.n=an+o1}.
\begin{lemma} \label{lem:density}
Let the boundary conditions~\eqref{eq:cond} be regular. Then for every $\eps > 0$ there exists $N_{\eps} > 0$ such that
\begin{equation} \label{eq:card.Re.ln0}
\card \left\{ n \in \bZ : |\Re \l_n^0| \le N \right\} \ge \frac{N}{\sigma + \eps}, \quad N \ge N_{\eps},
\qquad \sigma := \frac{\pi}{b_2-b_1} > 0.
\end{equation}
\end{lemma}
\begin{proof}
Asymptotic formula~\eqref{eq:lam.n=an+o1} for $\{\l_n^0\}_{n \in \bZ}$ implies that $|\Re \l_n^0| \le (2 \sigma + \eps) |n|$, $|n| \ge n_{\eps}$, for some $n_{\eps} \in \bN$. Hence
$$
\bZ \cap \(\[-\frac{N}{2\sigma + \eps}, -n_{\eps}\] \cup \[n_{\eps}, \frac{N}{2\sigma + \eps}\]\) \subset \cI_N := \left\{ n \in \bZ : |\Re \l_n^0| \le N \right\},
$$
for $N \ge (2\sigma+\eps) n_{\eps}$. Taking cardinalities in this inclusion implies
\begin{equation}
\card \cI_N \ge 2 \(\floor{\frac{N}{2\sigma + \eps}} - n_{\eps} + 1\)
\ge \frac{N}{\sigma + \eps/2} - 2 n_{\eps} \ge \frac{N}{\sigma + \eps},
\qquad N \ge N_{\eps},
\end{equation}
with $N_{\eps} := 2(\sigma/\eps+1)(2\sigma+\eps) n_{\eps}$.
\end{proof}
Combining two previous results leads to the following important property of zeros of characteristic determinant $\Delta_0(\cdot)$, which coincides with Proposition~\ref{prop:nlim.inf.intro} and is formulated again for reader's convenient.
\begin{proposition} \label{prop:nlim.inf}
Let $b_2/b_1 \notin \bQ$ and boundary conditions~\eqref{eq:cond.canon.intro} be regular, i.e. $u := ad-bc \ne 0$. Let $\{\l_n^0\}_{n \in \bZ}$ be the sequence of zeros of the characteristic determinant $\Delta_0(\cdot)$ counting multiplicity. Then each of the sequences $\{\exp(i b_1 \l_n^0)\}_{n \in \bZ}$ and $\{\exp(i b_2 \l_n^0)\}_{n \in \bZ}$ has infinite set of limit points.
\end{proposition}
\begin{proof}
\textbf{(i)} First, let $bc=0$. Then according to the proof of Lemma~\ref{lem:ln0.exp.asymp}, zeros of the characteristic determinant $\Delta_0(\cdot)$ are simple and split into two separated arithmetic progressions $\L_0^1 = \{\l_{1,n}^{0}\}_{n \in \bZ}$ and $\L_0^2 = \{\l_{2,n}^{0}\}_{n \in \bZ}$ given by~\eqref{eq:l1n.l2n.bc=0}. Let $k \in \{1,2\}$ and $j=2/k$. Since $E(z) = e^{2 \pi i z}$ is periodic with period 1, we have for $n \in \bZ$,
\begin{equation}
\exp(i b_k \l_{j,n}^0) = \exp\(2 \pi i n b_k/b_j + \omega_{k,j,a,d}\)
= \exp\(2 \pi i \fr{n b_k/b_j} + \omega_{k,j,a,d}\),
\end{equation}
where $\omega_{k,j,a,d}$ is an explicit constant that can be derived from~\eqref{eq:l1n.l2n.bc=0}. Since $b_k/b_j \notin \bQ$, then by the classical Kronecker theorem, the sequence $\curl{\fr{n b_k/b_j}}_{n \in \bZ}$ is everywhere dense on $[0,1]$. This implies that the sequence $\{\exp(i b_k \l_{j,n}^0)\}_{n \in \bZ}$ has infinite set of limit points, which finishes the proof in this case.

\textbf{(ii)} Now, let $bc \ne 0$ and assume the contrary: one of the sequences $\{\exp(i b_1 \l_n^0)\}_{n \in \bZ}$ and $\{\exp(i b_2 \l_n^0)\}_{n \in \bZ}$ has a finite set of limit points. For definiteness assume that
\begin{equation} \label{eq:nlim.e1n}
 \nlim \{\exp(i b_1 \l_n^0)\}_{n \in \bZ} = m_1 \in \bN.
\end{equation}
Recall that $e_{1n} := \exp(i b_1 \l_n^0)$ and $e_{2n} := \exp(-i b_2 \l_n^0)$, $n \in \bZ$, and also set
\begin{equation} \label{eq:ln0=an+ibn}
 \l_n^0 = \alp_n + i \beta_n, \qquad \alp_n := \Re \l_n^0,
 \quad \beta_n := \Im \l_n^0, \qquad n \in \bZ.
\end{equation}
It is clear that
\begin{equation} \label{eq:abs.e1n}
 |e_{1n}| = |\exp(i b_1 \l_n^0)| = \exp(-b_1 \Im \l_n^0)
 = \exp(-b_1 \beta_n), \qquad n \in \bZ
\end{equation}
It follows from~\eqref{eq:abs.e1n}, \eqref{eq:nlim.e1n}, \eqref{eq:e.bj.ln.asymp.1} and Lemma~\ref{lem:ln0.exp.asymp}(i), applied with $f_1(z) = -b_1^{-1} \log |z|$, that
\begin{equation} \label{eq:nlim.Imln0}
 \nlim \{\beta_n\}_{n \in \bZ} =
 \nlim \{-b_1^{-1} \log |e_{1n}| \}_{n \in \bZ} \le m_1.
\end{equation}
In turn, since $e_{1n} = |e_{1n}| e^{i b_1 \alp_n} \asymp 1$, $n \in \bZ$,
then by Lemma~\ref{lem:limit}(i) applied with $f_2(z) = \Im z / |z|$ we have
\begin{equation} \label{eq:nlim.sinb1}
 \nlim \curl{\sin\(b_1 \alp_n\)}_{n \in \bZ} =
 \nlim \curl{\Im e_{1n} / |e_{1n}|}_{n \in \bZ} \le m_1.
\end{equation}
Since $bc \ne 0$ and boundary conditions~\eqref{eq:cond.canon} are regular, then Lemma~\ref{lem:ln0.exp.asymp}(i) implies~\eqref{eq:1+ae1.1+de2}. Recall that $u := ad-bc \ne 0$. Since $\Delta(\l_n^0)=0$, $1 + a e_{1n} \ne 0$ and $1 + d e_{2n} \ne 0$, $n \in \bZ$, it follows from~\eqref{eq:Delta0.new} that for $n \in \bZ$:
\begin{equation} \label{eq:e1.via.e2}
 1 + d e_{2n} + a e_{1n} + u e_{1n} e_{2n} = 0, \qquad
 e_{2n} = - \frac{1 + a e_{1n}}{d + u e_{1n}}, \qquad
 e_{1n} = - \frac{1 + d e_{2n}}{a + u e_{2n}}.
\end{equation}
Since $e_{1n} \asymp 1$, $e_{2n} \asymp 1$, $n \in \bZ$, relations~\eqref{eq:1+ae1.1+de2} and~\eqref{eq:e1.via.e2} imply that
\begin{equation} \label{eq:a+ue2.d+ue1}
 d + u e_{1n} \asymp 1, \qquad a + u e_{2n} \asymp 1, \qquad n \in \bZ.
\end{equation}
Hence $f_3(z) := - \frac{1 + a z}{d + u z}$ is continuous in the neighborhood of $\{e_{1n}\}_{n \in \bZ}$. Combining this with Lemma~\ref{lem:limit}(i), the second identity in~\eqref{eq:e1.via.e2} and relation~\eqref{eq:nlim.e1n} we arrive at
\begin{equation} \label{eq:nlim.e2n}
 \nlim \{e_{2n}\}_{n \in \bZ} = \nlim \{f_3(e_{1n})\}_{n \in \bZ} \le m_1.
\end{equation}
Similarly to~\eqref{eq:nlim.sinb1} we get
\begin{equation} \label{eq:nlim.sinb2}
 \nlim \curl{\sin\(b_2 \alp_n\)}_{n \in \bZ} =
 \nlim \curl{\Im e_{2n} / |e_{2n}|}_{n \in \bZ} \le m_1.
\end{equation}
Since boundary conditions~\eqref{eq:cond.canon} are regular then Proposition~\ref{prop:sine.type} implies that the sequence $\{\alp_n\}_{n \in \bZ}$ is incompressible and Lemma~\ref{lem:density} implies the estimate~\eqref{eq:card.Re.ln0}, which in turn yields the estimate~\eqref{eq:card.alpn>} for $\{\alp_n\}_{n \in \bZ}$ with $\gamma = \frac{1}{2\sigma} = \frac{b_2-b_1}{2\pi}$. Since $b_2/b_1 \notin \bQ$ then by Lemma~\ref{lem:sin.weyl} one of sequences $\curl{\sin(b_1 \alp_n)}_{n \in \bZ}$ and $\curl{\sin(b_2 \alp_n)}_{n \in \bZ}$ has infinite set of limit points. This contradicts relations~\eqref{eq:nlim.sinb1} and~\eqref{eq:nlim.sinb2} and finishes the proof.
\end{proof}
\section{Bari $c_0$-property of the system of root vectors of the unperturbed operator} \label{sec:bari.c0}

In this section assuming boundary conditions~\eqref{eq:cond.canon} to be strictly regular, we show that the system of root vectors of the operator $L_U(0)$ is a Bari $c_0$-sequence in $\LLV{2}$ if and only if the operator $L_U(0)$ is selfadjoint. Since eigenfunctions of $L_U(0)$ in their ``natural form'' are not normalized we need the following simple practical criterion of Bari $c_0$-property.
\begin{lemma} \label{lem:crit.bari.not.norm}
Let $\fF = \{f_n\}_{n \in \bZ}$ be a complete minimal system of vectors in a Hilbert space $\fH$. Let also $\{g_n\}_{n \in \bZ}$ be ``almost biorthogonal'' to $\fF$. Namely, $(f_n, g_m) = 0$, $n \ne m$, $(f_n, g_n) \ne 0$, $n, m \in \bZ$. Then the normalized system
$$
\fF' := \curl{f_n'}_{n \in \bZ}, \quad f_n' := \frac{1}{\|f_n\|} f_n,
\quad n \in \bZ,
$$
is a Bari $c_0$-sequence in $\fH$ (see Definition~\ref{def:bari.c0}) if and only if
\begin{equation} \label{eq:sum.fn.gn-1}
 \frac{\|f_n\|_{\fH} \cdot \|g_n\|_{\fH}}{|(f_n, g_n)_{\fH}|} \to 1
 \quad\text{as}\quad n \to \infty.
\end{equation}
\end{lemma}
\begin{proof}
For brevity we set $\|\cdot\| := \|\cdot\|_{\fH}$ and $(\cdot,\cdot) := (\cdot,\cdot)_{\fH}$.
It is clear that for the system $\fG' := \{g_n'\}_{n \in \bZ}$ that is biorthogonal to $\fF'$ we have,
\begin{equation} \label{eq:fn'gm'}
 (f_n', g_m') = \delta_{nm}, \quad n,m \in \bZ; \qquad
 g_n' = \frac{\|f_n\|}{(f_n, g_n)} \cdot g_n, \quad n \in \bZ.
\end{equation}
Relations~\eqref{eq:fn'gm'} imply that
\begin{equation} \label{eq:|fn'-gn'|}
 \|f_n' - g_n'\|^2 = \|f_n'\|^2 - (f_n', g_n') - \ol{(f_n', g_n')}
 + \|g_n'\|^2 = \|g_n'\|^2 - 1
 = \frac{\|f_n\|^2 \cdot \|g_n\|^2}{|(f_n, g_n)|^2} - 1.
\end{equation}
Hence, systems $\fF'$ and $\fG'$ are $c_0$-close if and only if condition~\eqref{eq:sum.fn.gn-1} holds.
\end{proof}
The following simple property of compact operators with asymptotically simple spectrum will be also useful in the next section.
\begin{lemma} \label{lem:some.every}
Let $T$ be an operator with compact resolvent in a separable Hilbert space $\fH$ and let $\{\l_n\}_{n \in \bZ}$ be a sequence of its eigenvalues counting multiplicities. Let also $p \in [1,2]$.
Assume that for some $N \in \bN$ eigenvalues $\l_n$, $|n| \ge N$, are algebraically simple.
Then if some normalized system of root vectors of the operator $T$ is a Bari $(\ell^p)^*$-sequence in $\fH$ then every normalized system of root vectors of the operator $T$ is a Bari $(\ell^p)^*$-sequence in $\fH$.
\end{lemma}
\begin{proof}
Let $\fF = \{f_n\}_{n \in \bZ}$
be a normalized system of root vectors of the operator $T$, which is a Bari $(\ell^p)^*$-sequence in $\fH$. By definition,
the system $\fF$ is complete and minimal in $\fH$. Let $\fG = \{g_n\}_{n \in \bZ}$ be its (unique) biorthogonal system. Further, let $\fF' = \{f_n'\}_{n \in \bZ}$ be any other normalized system of root vectors of the operator $T$. Since eigenvalue $\l_n$, $|n| \ge N$, is algebraically simple then $\dim(\cR_{\l_n}(T)) = 1$, $|n| \ge N$. Hence $f_n' = \alp_n f_n$, $|n| \ge N$, for some $\alp_n \in \bT := \{z \in \bC : |z|=1\}$. It is clear that $\fF'$ is also complete and minimal and for its biorthogonal sysyem $\fG' = \{g_n'\}_{n \in \bZ}$ we have that $g_n' = \ol{\alp_n^{-1}} g_n = \alp_n g_n$, $|n| \ge N$, since $|\alp_n|=1$. Hence $\|f_n'-g_n'\|_{\fH} = \|\alp_n \cdot (f_n-g_n)\|_{\fH} = \|f_n-g_n\|_{\fH}$, $|n| \ge N$. This implies that $\curl{\|f_n'-g_n'\|_{\fH}}_{n \in \bZ} = \curl{\|f_n-g_n\|_{\fH}}_{n \in \bZ} \in (\ell^p)^*$ and finishes the proof.
\end{proof}
\begin{remark} \label{rem:sa.c0.bari}
Let $A$ be a selfadjoint operator with compact resolvent. Then every its normalized system of root vectors is an orthonormal basis in $\LLV{2}$ and coincides with its biorthogonal sequence. This implies that every normalized system of root vectors of the operator $A$ is a Bari $c_0$-sequence.
\end{remark}
To study norms $\|f_n^0\|_2$ and $\|g_n^0\|_2$ of the eigenvectors of the operators $L_U(0)$ and $L_U^*(0)$, we first need to obtain some properties of simple integrals $\int_0^1 |e^{\pm 2 i b_j \l x}| dx$, $j \in \{1,2\}$, $\l \in \bC$.
\begin{lemma} \label{lem:ejx.ej.Ej}
Denote for $j \in \{1,2\}$ and $\l \in \bC$:
\begin{equation} \label{eq:Ejpm.def}
 E_{j}^{\pm}(\l) := \int_0^1 \abs{e^{\pm 2 i b_j \l x}} dx
 = \int_0^1 e^{\mp 2 b_j \Im \l x} dx
 = \frac{e^{\mp 2 b_j \Im \l} - 1}{\mp 2 b_j \Im \l}.
\end{equation}
Then the following estimate holds:
\begin{align}
\label{eq:Ej+Ej->1}
 & E_j^+(\l) E_j^-(\l) - 1 \ge \frac{(b_j \Im \l)^2}{3},
 \qquad j \in \{1,2\}, \quad \l \in \bC.
\end{align}
In particular, $E_j^+(\l) E_j^-(\l) - 1 > 0$ if $\Im \l \ne 0$.
\end{lemma}
\begin{proof}
Let $h \ge 0$. It is clear that
\begin{equation} \label{eq:Ej=f}
 E_j^{\pm}(\l) = f(\mp 2 b_j \Im \l), \quad\text{where}\quad
 f(x) := \frac{e^x - 1}{x} = 1 + \frac{x}{2} + O(x^2), \quad |x| < h.
\end{equation}
It follows from Taylor expansion of $e^x$ that for $x \in \bR$:
\begin{equation} \label{eq:fx.f-x>1}
 f(x)f(-x) = \frac{e^x - 1}{x} \cdot \frac{e^{-x} - 1}{-x}
 = \frac{e^x + e^{-x} - 2}{x^2}
 = 2 \sum_{k=1}^{\infty} \frac{x^{2k-2}}{(2k)!}
 \ge 1 + \frac{x^2}{12}.
\end{equation}
Estimate~\eqref{eq:Ej+Ej->1} now immediately follows from~\eqref{eq:Ej=f} and~\eqref{eq:fx.f-x>1}.

\end{proof}
First we establish the Bari $c_0$-property criterion in a special case $b=c=0$.
\begin{proposition} \label{prop:crit.bari.period}
Let boundary conditions~\eqref{eq:cond.canon} be strictly regular with $b=c=0$, i.e. they are of the form
\begin{equation} \label{eq:quasi.per.bc}
 y_1(0) + a y_1(1) = d y_2(0) + y_2(1) = 0, \qquad a d \ne 0.
\end{equation}
Then some normalized system of root vectors of the operator $L_U(0)$ is a Bari $c_0$-sequence in $\LLV{2}$ (see Definition~\ref{def:bari.c0}) if and only if $|a| = |d| = 1$.
\end{proposition}
\begin{proof}
\textbf{(i)} If $|a|=|d|=1$ (and $b=c=0$) then by Corollary~\ref{cor:sa.crit} the operator $L_U(0)$ with boundary conditions~\eqref{eq:quasi.per.bc} is self-adjoint. Remark~\ref{rem:sa.c0.bari} now finishes the proof.

\textbf{(ii)} Now assume that some normalized system of root vectors of the operator $L_U(0)$ is a Bari $c_0$-sequence in $\LLV{2}$. Since boundary conditions~\eqref{eq:quasi.per.bc} are strictly regular then by definition, eigenvalues of the operator $L_U(0)$ are asymptotically simple.
Hence by Lemma~\ref{lem:some.every} every normalized system of root vectors of the operator $L_U(0)$ is a Bari $c_0$-sequence in $\LLV{2}$.

According to the proof of Lemma~\ref{lem:ln0.exp.asymp} the eigenvalues of the operator $L_U(0)$ are simple and split into two separated arithmetic progressions $\L_0^1 = \{\l_{1,n}^{0}\}_{n \in \bZ}$ and $\L_0^2 = \{\l_{2,n}^{0}\}_{n \in \bZ}$ given by~\eqref{eq:l1n.l2n.bc=0}.
It is easy to verify that the vectors
\begin{equation}
 f_{1,n}^0(x) = \binom{e^{i b_1 \l_{1,n}^0 x}}{0}, \qquad
 g_{1,n}^0(x) = \binom{e^{i b_1 \ol{\l_{1,n}^0} x}}{0},
 \qquad n \in \bZ,
\end{equation}
are the eigenvectors of the operators $L_U(0)$ and $L_U^*(0)$ corresponding to the eigenvalues $\l_{1,n}^0$ and $\ol{\l_{1,n}^0}$, and the vectors
\begin{equation}
 f_{2,n}^0(x) = \binom{0}{e^{i b_2 \l_{2,n}^0 x}}, \qquad
 g_{2,n}^0(x) = \binom{0}{e^{i b_2 \ol{\l_{2,n}^0} x}}, \qquad n \in \bZ,
\end{equation}
are the eigenvectors of the operators $L_U(0)$ and $L_U^*(0)$  corresponding to the eigenvalues $\l_{2,n}^0$ and $\ol{\l_{2,n}^0}$ respectively. It is clear that
\begin{equation} \label{eq:fjn.gkm=delta}
 \(f_{j,n}^0, g_{k,m}^0\)_2 = \delta_{j,n}^{k,m}, \qquad j,k \in \{1,2\},
 \quad n,m \in \bZ.
\end{equation}
Thus the union system $\fF := \{f_{1,n}^0\}_{n \in \bZ} \cup \{f_{2,n}^0\}_{n
\in \bZ}$ is the system of root vectors of the operator $L_U(0)$ and $\fG :=
\{g_{1,n}^0\}_{n \in \bZ} \cup \{g_{2,n}^0\}_{n \in \bZ}$ is biorthogonal to it. Hence normalization of the system $\fF$ is a Bari $c_0$-sequence in $L^2([0,1]; \bC^2)$. According to Lemma~\ref{lem:crit.bari.not.norm} we have
\begin{equation} \label{eq:sumj.sumn.alp}
 \alp_{j,n} := \frac{\bigl\|f_{j,n}^0\bigr\|_2 \cdot
 \bigl\|g_{j,n}^0\bigr\|_2}{\abs{\bigl(f_{j,n}^0, g_{j,n}^0\bigr)_2}}
 \to 1 \quad\text{as}\quad n \to \infty, \qquad j \in \{1,2\}.
\end{equation}
Let $j=1$. Then taking into account Lemma~\ref{lem:ejx.ej.Ej} and formula~\eqref{eq:fjn.gkm=delta} we have
\begin{equation} \label{eq:alp1n>Im}
 \alp_{1,n}^2 = \bigl\|f_{1,n}^0\bigr\|_2^2 \cdot
 \bigl\|g_{1,n}^0\bigr\|_2^2 = E_1^+(\l_{1,n}^0) E_1^-(\l_{1,n}^0)
 \ge 1 + \frac13 \bigabs{b_2 \Im \l_{1,n}^0}^2, \qquad n \in \bZ.
\end{equation}
It follows from~\eqref{eq:l1n.l2n.bc=0} that $b_1 \Im \l_{1,n}^0 = \ln|a|$. Since $\alp_{1,n} \to 1$ as $n \to \infty$, formula~\eqref{eq:alp1n>Im} implies that $\ln|a|=0$, or $|a|=1$.

Similarly considering the case $j=2$ we conclude that $|d|=1$, which finishes the proof.
\end{proof}
In the following intermediate result we reduce condition~\eqref{eq:sum.fn.gn-1} of Bari $c_0$-property of the system of root vectors of the operator $L_U(0)$ to explicit condition in terms of eigenvalues $\{\l_n^0\}_{n \in \bZ}$.
Recall that $\beta = -b_2/b_1 > 0$.
\begin{proposition} \label{prop:crit.bari.b.ne.0}
Let boundary conditions~\eqref{eq:cond.canon} be strictly regular and let one of the parameters $b$ or $c$ in them be non-zero, $|b|+|c|>0$. Let $\{\l_n^0\}_{n \in \bZ}$ be the sequence of the eigenvalues of the operator $L_U(0)$ counting multiplicities. Then some normalized system of root vectors of the operator $L_U(0)$ is a Bari $c_0$-sequence in $\LLV{2}$ (see Definition~\ref{def:bari.c0}) if and only if the following conditions hold
\begin{equation} \label{eq:lim1.lim2}
 |c| = \beta |b|, \qquad \lim_{n \to \infty} \Im \l_n^0 = 0
 \quad\text{and}\quad \lim_{n \to \infty} z_n = |bc|,
\end{equation}
where
\begin{equation} \label{eq:zn.def}
 z_n := \(1 + d e^{- i b_2 \l_n^0}\)\ol{\(1 + a e^{i b_1 \l_n^0}\)}.
\end{equation}

\end{proposition}
\begin{proof}
Without loss of generality we can assume that $b \ne 0$. By definition of strictly regular boundary conditions there exists $n_0 \in \bN$ such that eigenvalues $\l_n^0$ of $L_U(0)$ for $|n| > n_0$ are algebraically simple and separated from each other. According to the proof of Theorem~1.1 in~\cite{LunMal16JMAA} vector-functions $f_n^0(\cdot)$ and $g_n^0(\cdot)$, $|n| > n_0$, of the following form:
\begin{equation} \label{eq:fn0x.gn0x}
 f_n^0(x) := \binom{b e^{i b_1 \l_n^0 x}}{
 - (1 + a e^{i b_1 \l_n^0}) e^{i b_2 \l_n^0 x}}, \qquad
 g_n^0(x) := \ol{\binom{(1 + d e^{-i b_2 \l_n^0})
 e^{-i b_1 \l_n^0 x}}{ - \beta b e^{-i b_2 \l_n^0 x}}},
\end{equation}
are non-zero eigenvectors of the operators $L_U(0)$ and $L_U^*(0)$ corresponding to the eigenvalues $\l_n^0$ and $\ol{\l_n^0}$ for $|n| > n_0$, respectively. Let $f_n^0(\cdot)$ and $g_n^0(\cdot)$ be some root vectors of operators $L_U(0)$ and $L_U^*(0)$ corresponding to the eigenvalues $\l_n^0$ and $\ol{\l_n^0}$ for $|n| \le n_0$. Clearly $\fF := \{f_n^0\}_{n \in \bZ}$ is a system of root vectors of the operator $L_U(0)$ and $\fG := \{g_n^0\}_{n \in \bZ}$ is the corresponding system for the adjoint operator $L_U^*(0)$. Let us show that normalization of $\fF$ is a Bari $c_0$-sequence in $\LLV{2}$ if and only condition~\eqref{eq:lim1.lim2} holds. Since eigenvalues of the operator $L_U(0)$ are asymptotically simple Lemma~\ref{lem:some.every} will imply the statement of the proposition.
Clearly, $\fG$ is almost biorthogonal to $\fF$. Hence Lemma~\ref{lem:crit.bari.not.norm} implies that normalization of $\fF$ is a Bari $c_0$-sequence in $\LLV{2}$ if and only if condition~\eqref{eq:sum.fn.gn-1} holds.

Set for brevity $E_{jn}^{\pm} := E_j^{\pm}(\l_n^0)$, $j \in \{1,2\}$, $n \in \bZ$, where $E_j^{\pm}(\l)$ is defined in~\eqref{eq:Ejpm.def}. With account of this notation and notation~\eqref{eq:ekn.def} we get after performing straightforward calculations:
\begin{align}
\label{eq:|fn|2}
 \|f_n^0\|_2^2 &= |b|^2 E_{1n}^+ + |1 + a e_{1n}|^2 E_{2n}^+, \\
\label{eq:|gn|2}
 \|g_n^0\|_2^2 &= \abs{1 + d e_{2n}}^2 E_{1n}^- + \beta^2 |b|^2 E_{2n}^-, \\
\label{eq:fn.gn}
 (f_n^0, g_n^0)_2 &= b\((1 + d e_{2n}) + \beta (1 + a e_{1n})\).
\end{align}
Since boundary conditions~\eqref{eq:cond.canon} are strictly regular, it follows from the proof of Theorem~1.1 in~\cite{LunMal16JMAA} that the following estimate holds
\begin{equation}
 (f_n^0, g_n^0)_2 \asymp \Delta'(\l_n^0) \asymp 1, \quad |n| > n_0.
\end{equation}
Hence condition~\eqref{eq:sum.fn.gn-1} is equivalent to
\begin{equation} \label{eq:fn.gn-fngn.to0}
 \|f_n^0\|_2^2 \cdot \|g_n^0\|_2^2 - |(f_n^0, g_n^0)_2|^2
 \to 0 \quad\text{as}\quad n \to \infty.
\end{equation}
With account of~\eqref{eq:|fn|2}--\eqref{eq:fn.gn} we get
\begin{multline} \label{eq:fn2.gn2-fn.gn2=tau.sum}
 \|f_n^0\|_2^2 \cdot \|g_n^0\|_2^2 - |(f_n^0, g_n^0)_2|^2
 = \(|b|^2 \cdot E_{1n}^+ + |1 + a e_{1n}|^2 \cdot E_{2n}^+\) \cdot
 \(\abs{1 + d e_{2n}}^2 E_{1n}^- + \beta^2 |b|^2 E_{2n}^-\) \\
 - |b|^2 \abs{(1 + d e_{2n}) + \beta (1 + a e_{1n})}^2
 = \tau_{1,n} + \tau_{2,n} + \tau_{3,n}, \qquad |n| > n_0,
\end{multline}
where
\begin{align}
\label{eq:tau1}
 \tau_{1,n} &:= |b|^2 \cdot \abs{1 + d e_{2n}}^2 \cdot
 (E_{1n}^+ E_{1n}^- - 1), \\
\label{eq:tau2}
 \tau_{2,n} &:= \beta^2 |b|^2 \cdot |1 + a e_{1n}|^2 \cdot (E_{2n}^+ E_{2n}^- - 1), \\
\label{eq:tau3}
 \tau_{3,n} &:= \beta^2 |b|^4 E_{1n}^+ E_{2n}^- +
 |z_n|^2 \cdot E_{2n}^+ E_{1n}^- -
 2 \beta |b|^2 \cdot \Re z_n,
\end{align}
where $z_n = (1 + d e_{2n}) \ol{(1 + a e_{1n})}$ is defined in~\eqref{eq:zn.def}. According to Proposition~\ref{prop:sine.type}, $|\Im \l_n^0| \le h$, $n \in \bZ$, for some $h \ge 0$. Hence terms $|1 + d e_{2n}|$, $|1 + a e_{1n}|$, $|z_n|$, $E_{1n}^{\pm}$ and $E_{2n}^{\pm}$ are all bounded for $n \in \bZ$.

First assume that $\Im \l_n^0 \to 0$ as $n \to \infty$. Then it is clear from~\eqref{eq:Ej=f} that
\begin{equation} \label{eq:ejnto1}
 |e_{jn}| \to 1 \quad\text{and}\quad
 E_{jn}^{\pm} \to 1  \quad\text{as}\quad n \to \infty,
 \qquad j \in \{1,2\}.
\end{equation}
Hence $\tau_{1,n} + \tau_{2,n} \to 0$ as $n \to \infty$, while
$\tau_{3,n} \to 0$ as $n \to \infty$ if and only if
\begin{equation} \label{eq:tau4n.def}
 \tau_{4,n} :=
 \beta^2 |b|^4 + |z_n|^2
 - 2 \beta |b|^2 \cdot \Re z_n = \abs{z_n - \beta |b|^2}^2 \to 0 \quad\text{as}\quad n \to \infty.
\end{equation}
It follows from~\eqref{eq:Delta_0_in_roots} and~\eqref{eq:ejnto1} that $|z_n| = |bc|\cdot|e_{1n} e_{2n}| \to |bc|$ as $n \to \infty$. Hence, since $b \ne 0$, then
\begin{equation} \label{eq:tau4n.to0}
 \Bigl(\tau_{4,n} \to 0 \ \ \text{as}\ \ n \to \infty\Bigr)
 \ \ \Leftrightarrow \ \
 \Bigl(|c| = \beta |b| \ \ \text{and}\ \ z_n \to |bc|
 \ \ \text{as}\ \ n \to \infty\Bigr).
\end{equation}
Now if condition~\eqref{eq:lim1.lim2} holds then~\eqref{eq:tau4n.to0} and previous observations on $\tau_{1,n}$, $\tau_{2,n}$, $\tau_{3,n}$, $\tau_{4,n}$ imply the desired condition~\eqref{eq:fn.gn-fngn.to0}.

Now assume that condition~\eqref{eq:fn.gn-fngn.to0} holds.
It follows from~\eqref{eq:Ej+Ej->1} and~\eqref{eq:Ej=f} that
\begin{equation} \label{eq:Ej.Ej-1.asymp}
 0 \le E_{jn}^+ E_{jn}^- - 1 \asymp |\Im \l_n^0|^2, \quad n \in \bZ,
 \qquad j \in \{1,2\}.
\end{equation}
Since $b \ne 0$ and $\beta > 0$ relations~\eqref{eq:tau1}--\eqref{eq:tau2} and~\eqref{eq:Ej.Ej-1.asymp}
combined with Lemma~\ref{lem:ln0.exp.asymp} imply that
\begin{equation} \label{eq:tau1+tau2.asymp}
 \tau_{1,n} \ge 0, \quad \tau_{2,n} \ge 0, \quad \tau_{1,n} + \tau_{2,n} \asymp |\Im \l_n^0|^2, \quad |n| > n_0.
\end{equation}
With account of~\eqref{eq:Ej.Ej-1.asymp} we get for $n \in \bZ$:
\begin{equation} \label{eq:tau3.estim}
 \beta^2 |b|^4 E_{1n}^+ E_{2n}^- + |z_n|^2
 \cdot E_{2n}^+ E_{1n}^-
 \ge 2 \beta |b|^2 |z_n|
 \sqrt{E_{1n}^+ E_{1n}^- \cdot E_{2n}^- E_{2n}^+} \nonumber \\
 \ge 2 \beta |b|^2 \cdot \Re z_n.
\end{equation}
Hence $\tau_{3,n} \ge 0$, $n \in \bZ$. Since $\tau_{1,n} + \tau_{2,n} + \tau_{3,n} = \|f_n^0\|_2^2 \cdot \|g_n^0\|_2^2 - |(f_n^0, g_n^0)|_2^2 \to 0$ as $n \to \infty$, then $\tau_{j,n} \to 0$ as $n \to \infty$, $j \in \{1,2,3\}$. Then condition~\eqref{eq:tau1+tau2.asymp} implies that $\Im \l_n^0 \to 0$ as $n \to \infty$. Combining this with the fact that $\tau_{3,n} \to 0$ as $n \to \infty$, implies that $\tau_{4,n} \to 0$ as $n \to \infty$, where $\tau_{4,n}$ is defined in~\eqref{eq:tau4n.def}. Now, equivalence~\eqref{eq:tau4n.to0} finishes the proof.
\end{proof}
In the next result we reduce part of the condition~\eqref{eq:lim1.lim2} to an explicit condition on the coefficients $a, b, c, d$ in the boundary conditions~\eqref{eq:cond.canon} in the difficult case $b_2/b_1 \notin \bQ$.
\begin{lemma} \label{lem:lim.Imln}
Let boundary conditions~\eqref{eq:cond.canon} be regular, i.e. $u := ad-bc \ne 0 $. Let also $\beta = -b_2/b_1 \notin \bQ$. Let $\{\l_n^0\}_{n \in \bZ}$ be a sequence of zeros of the characteristic determinant $\Delta_0(\cdot)$ counting multiplicity. Let
\begin{equation} \label{eq:lim.Imln.zn}
 bc \ne 0, \qquad |a|+|d|>0, \qquad \Im \l_n^0 \to 0
 \quad\text{and}\quad
 z_n \to |bc|  \quad\text{as}\quad n \to \infty,
\end{equation}
where $z_n$ is defined in~\eqref{eq:zn.def}.
Then
\begin{equation} \label{eq.ad-bc=d/a}
 |a|=|d|>0, \qquad u = ad - bc = d/\ol{a}
 \qquad\text{and}\qquad ad\ol{bc} < 0.
\end{equation}
\end{lemma}
\begin{proof}
Since $\Im \l_n^0 \to 0$ as $n \to \infty$ then $|e_{1n}| \to 1$ as $n \to \infty$.
Further, since boundary conditions are regular, $b_2/b_1 \notin \bQ$ and $bc \ne 0$ then all considerations in the proof of Proposition~\ref{prop:nlim.inf} are valid. Since $u-ad=-bc$, then the second relation in~\eqref{eq:e1.via.e2} implies:
\begin{multline} \label{eq:zn=frac.e1}
 z_n = (1 + d e_{2n}) \cdot \ol{(1+a e_{1n})}
 = \(1 - d \frac{1 + a e_{1n}}{d + u e_{1n}}\) \cdot
 (1 + \ol{a} \ol{e_{1n}}) \\
 = \frac{-bc e_{1n} (1 + \ol{a} \ol{e_{1n}})}{d + u e_{1n}}
 = \frac{-bc (e_{1n} + \ol{a} |e_{1n}|^2)}{d + u e_{1n}},
 \qquad n \in \bZ.
\end{multline}
Recall that $d + u e_{1n} \asymp 1$, $n \in \bZ$, as established in~\eqref{eq:a+ue2.d+ue1}. Since $z_n \to |bc|$ and $|e_{1n}| \to 1$ as $n \to \infty$, then~\eqref{eq:zn=frac.e1} implies that
\begin{equation}
 |bc| (d + u e_{1n}) + bc (e_{1n} + \ol{a}) \to 0
 \quad\text{as}\quad n \to \infty,
\end{equation}
or
\begin{equation} \label{eq:f.e1.lim}
 (|bc| u + bc) e_{1n} + |bc| d + bc \ol{a} \to 0
 \quad\text{as}\quad n \to \infty,
\end{equation}
But by Proposition~\ref{prop:nlim.inf} the sequence $\{e_{1n}\}_{n \in \bZ}$ has infinite set of limit points. Hence relation~\eqref{eq:f.e1.lim} is possible only if
\begin{equation} \label{eq:ubc.ua=d}
 |bc| u = -bc \quad\text{and}\quad |bc| d = -bc \ol{a}.
\end{equation}
Since $|a|+|d|>0$ and $bc\ne 0$, then the second relation in~\eqref{eq:ubc.ua=d} implies that $|a| = |d| > 0$ and that $bc \ol{ad} = - |bc| |d|^2 < 0$. This implies the first and the third relations in~\eqref{eq.ad-bc=d/a}. Further, combining both relations in~\eqref{eq:ubc.ua=d} implies the second relation in~\eqref{eq.ad-bc=d/a}: $u = d/\ol{a} = -bc/|bc|$, which finishes the proof.
\end{proof}
Now we are ready to state the main result of this section.
\begin{theorem} \label{th:crit.c0.bari}
Let boundary conditions~\eqref{eq:cond.canon} be strictly regular. Then some normalized system of root vectors of the operator $L_U(0)$ is a Bari $c_0$-sequence in $\LLV{2}$ (see Definition~\ref{def:bari.c0}) if and only the operator $L_U(0)$ is self-adjoint. The latter holds if and only if coefficients $a,b,c,d$ from boundary conditions~\eqref{eq:cond.canon} satisfy the following relations:
\begin{equation} \label{eq:abcd.sa}
 |a|^2 + \beta |b|^2 = 1, \qquad
 |c|^2 + \beta |d|^2 = \beta, \qquad
 a \ol{c} + \beta b \ol{d} = 0, \qquad \beta := -b_2/b_1 > 0.
\end{equation}
In this case every normalized system of root vectors of the operator $L_U(0)$ is a Bari $c_0$-sequence in $\LLV{2}$.
\end{theorem}
\begin{proof}

\textbf{(i)} If conditions~\eqref{eq:abcd.sa} hold then by Corollary~\ref{cor:sa.crit} the operator $L_U(0)$ with boundary conditions~\eqref{eq:quasi.per.bc} is self-adjoint. Remark~\ref{rem:sa.c0.bari} now finishes the proof.

\textbf{(ii)} Now assume that some normalized system of root vectors of the operator $L_U(0)$ forms a Bari basis in $\LLV{2}$.

If $b=c=0$ then Proposition~\ref{prop:crit.bari.period} yields that $|a|=|d|=1$, in which case operator $L_U(0)$ is self-adjoint. This finishes the proof in this case.

Now let $|b|+|c| \ne 0$. Proposition~\ref{prop:crit.bari.b.ne.0} implies that relations~\eqref{eq:lim1.lim2} take place. In particular, $|c| = \beta |b|$. Consider three cases.

\textbf{Case A.} Let $b_1 / b_2 \in \bQ$. In this case $b_1 = -m_1 b_0$, $b_2 = m_2 b_0$, where $b_0 > 0$, $m_1, m_2 \in \bN$. Set $m = m_1 + m_2$. Since $ad \ne bc$, then $\Delta_0(\l) e^{-i b_1 \l}$ is a polynomial in $e^{i b_0 \l}$ of degree $m$ with non-zero roots $e^{i\mu_k}$, $\mu_k \in \bC$, $k \in \{1, \ldots, m\}$, counting multiplicities. Hence, the sequence of zeros $\{\l_n^0\}_{n \in \bZ}$ of $\Delta_0(\cdot)$ is a union of arithmetic progressions $\left\{\frac{\mu_k + 2 \pi n}{b_0}\right\}_{n \in \bZ}$, $k \in \{1, \ldots, m\}$. Clearly $\Im \l_n^0 = \Im \mu_{k_n} / b_0$ for some $k_n \in \{1, \ldots, m\}$. It is clear, that if $\Im \mu_k \ne 0$, for some $k \in \{1, \ldots, m\}$, then $\Im \l_n^0$ does not tend to $0$ as $n \to \infty$.
Hence $\Im \l_n^0 = 0$, $n \in \bZ$. This implies that
\begin{equation} \label{eq:Ejn=1}
 E_{jn}^{\pm} = \int_0^1 \abs{e^{\pm 2 i b_j \l_n^0 x}} dx = 1,
 \qquad n \in \bZ, \qquad j \in \{1,2\}.
\end{equation}
It is clear that $e^{-i b_2 \l_n^0} = (e^{-i \mu_{k_n}})^{m_2}$ for some $k_n \in \{1, \ldots, m\}$, $n \in \bZ$, and $\{k_n\}_{n \in \bZ}$ is a periodic sequence. Hence the sequence $\{e^{-i b_2 \l_n^0}\}_{n \in \bZ}$ is periodic. Similarly the sequence $\{e^{i b_1 \l_n^0}\}_{n \in \bZ}$ is periodic. Hence, the sequence
\begin{equation}
\{z_n\}_{n \in \bZ}, \qquad z_n = \(1 + d e^{- i b_2 \l_n^0}\)\ol{\(1 + a e^{i b_1 \l_n^0}\)} = \(1 + d e_{2n}\)\ol{\(1 + a e_{1n}\)},
\end{equation}
is periodic. Since $z_n \to |bc|$ as $n \to \infty$ and $|c| = \beta|b|$, it implies that $z_n = |bc| = \beta|b|^2$, $n \in \bZ$. It now follows from~\eqref{eq:fn2.gn2-fn.gn2=tau.sum}--\eqref{eq:tau4n.def} and~\eqref{eq:Ejn=1} that
\begin{equation}
\|f_n^0\|_2^2 \cdot \|g_n^0\|_2^2 - |(f_n^0, g_n^0)|_2^2 = \tau_{4,n} =
\abs{z_n - \beta |b|^2}^2 = 0, \qquad n \in \bZ.
\end{equation}
Taking into account formula~\eqref{eq:|fn'-gn'|} we see that the normalized eigenvectors $f_n'$ and $g_n'$ of the operators $L_U(0)$ and $L_U^*(0)$ corresponding to the common eigenvalue $\l_n^0 = \ol{\l_n^0}$ are equal for all $n \in \bZ$, which implies that $L_U(0) = L_U^*(0)$.

\textbf{Case B.} Let $a=d=0$.
Then $z_n = 1$, $n \in \bZ$. Since $z_n \to |bc|$ as $n \to \infty$, then $|bc| = 1$. Combined with $|c| = \beta|b|$, this implies the desired condition~\eqref{eq:abcd.sa}, and finishes the proof in this case.

\textbf{Case C.} Finally, let $b_1 / b_2 \notin \bQ$, $|a|+|d|>0$ and $bc \ne 0$. Since $\Im \l_n^0 \to 0$ and $z_n \to |bc|$ as $n \to \infty$, then Lemma~\ref{lem:lim.Imln} implies condition~\eqref{eq.ad-bc=d/a}. In particular, $|a|=|d|>0$ and $ad\ol{bc} < 0$. Since, in addition, $|c| = \beta |b| > 0$, then
\begin{equation} \label{eq:adbc=b2d2}
 -ad\ol{bc} = |ad \cdot bc| = |d|^2 \cdot \beta |b|^2 = \beta b \ol{d} \cdot d \ol{b}.
\end{equation}
Since $d \ol{b} \ne 0$, this implies that $a \ol{c} + \beta b \ol{d} = 0$ and coincides with the third condition in~\eqref{eq:abcd.sa}. Further, the second relation in~\eqref{eq.ad-bc=d/a}, combined with relations~\eqref{eq:adbc=b2d2}, $|a|=|d|$ and $|c| = \beta |b|$, implies that
\begin{multline} \label{eq:0=a+b=c+d}
 0 = \(-ad + bc + d / \ol{a}\)\ol{bc} =
 -ad\ol{bc} + |bc|^2 + \frac{ad\ol{bc}}{|a|^2}
 = |ad|\cdot|bc| + |bc|^2 - \frac{|ad|\cdot|bc|}{|a|^2} \\
 = |bc|(|ad| + |bc|-1)
 = |bc|(|a|^2 + \beta|b|^2-1) = |bc|(|d|^2 + \beta^{-1}|c|^2-1).
\end{multline}
Since $bc \ne 0$, relation~\eqref{eq:0=a+b=c+d} implies the first and second relations in~\eqref{eq:abcd.sa}, which finishes the proof.
\end{proof}
\section{The proof of the main result}
This section is devoted to the proof of the main result of the paper, Theorem~\ref{th:crit.lp.bari}. Throughout the section we use the following notations:
\begin{equation}
\fH := \LLV{2}, \qquad
\|\cdot\| := \|\cdot\|_2 = \|\cdot\|_{\fH} \quad\text{and}\quad
(\cdot,\cdot) := (\cdot,\cdot)_{2} = (\cdot,\cdot)_{\fH}.
\end{equation}

First we need the following trivial corollary from Theorem~\ref{th:ellp-close}.
\begin{corollary} \label{cor:every.SRV}
Let $Q \in \LL{p}$ for some $p \in [1,2]$ and boundary conditions~\eqref{eq:cond} be strictly regular. Let $\fF := \{f_n\}_{n \in \bZ}$ be a system of root vectors of the operator $L_U(Q)$ such that $\|f_n\| \asymp 1$, $n \in \bZ$. Then there exists a system of root vectors $\fF_0 := \{f_n^0\}_{n \in \bZ}$ of the operator $L_U(0)$ such that $\curl{\|f_n - f_n^0\|}_{n \in \bZ} \in (\ell^{p})^*$ and $\|f_n\| = \|f_n^0\|$, $n \in \bZ$.
\end{corollary}
\begin{proof}
Combining relations~\eqref{eq:sum.fn-fn0} and~\eqref{eq:lim.fn-fn0.c0} from Theorem~\ref{th:ellp-close} applied with $\wt{Q}=0$, implies existence of normalized systems of root vectors $\wt{\fF} := \{\wt{f}_n\}_{n \in \bZ}$ and $\wt{\fF}_0 := \{\wt{f}_n^0\}_{n \in \bZ}$ of the operators $L_U(Q)$ and $L_U(0)$, respectively, such that $\curl{\|\wt{f}_n - \wt{f}_n^0\|_{\infty}}_{n \in \bZ} \in (\ell^{p})^*$. Hence
$\curl{\|\wt{f}_n - \wt{f}_n^0\|}_{n \in \bZ} \in (\ell^{p})^*$. By Proposition~\ref{prop:Delta.regular.basic} eigenvalues of $L_U(Q)$ are asymptotically simple. Hence vectors $f_n$ and $\wt{f}_n$, $|n| \ge N$, are proportional for some $N \in \bN$, i.e. $f_n = \alp_n \wt{f}_n$, $|n| \ge N$, for some $\alp_n \in \bC$.
Let us set
\begin{equation}
 \fF_0 := \{f_n^0\}_{n \in \bZ}, \qquad
 f_n^0 := \begin{cases}
  \alp_n \wt{f}_n^0, & |n| \ge N, \\
  \|f_n\| \wt{f}_n^0, & |n| < N. \\
 \end{cases}
\end{equation}
It is clear that $\fF_0$ is a system of root vectors of the operator $L_U(0)$ and $\|f_n^0\| = \|f_n\|$, $n \in \bZ$. Moreover, $\|f_n - f_n^0\| = |\alp_n| \cdot \|\wt{f}_n - \wt{f}_n^0\|$, $|n| \ge N$. Since $|\alp_n| = \|f_n\| \asymp 1$, $|n| \ge N$, and $\curl{\|\wt{f}_n - \wt{f}_n^0\|}_{n \in \bZ} \in (\ell^{p})^*$, then $\curl{\|f_n - f_n^0\|}_{n \in \bZ} \in (\ell^{p})^*$, which finishes the proof.
\end{proof}
Now we are ready to prove our main result on Bari $(\ell^p)^*$-property.
\begin{proof}[Proof of Theorem~\ref{th:crit.lp.bari}]
Recall that $Q \in \LL{p}$ for some $p \in[1,2]$. Also note that if $L_U(0)$ is selfadjoint then Theorem~\ref{th:crit.c0.bari} implies conditions~\eqref{eq:abcd.sa.intro} on the coefficients from boundary conditions~\eqref{eq:cond.canon}.

\textbf{(i)} First assume that the operator $L_U(0)$ is selfadjoint and let $\fF := \{f_n\}_{n \in \bZ}$ be some normalized system of root vector of the operators $L_U(Q)$. By Corollary~\ref{cor:every.SRV} there exists normalized system of root vectors $\fF_0 := \{f_n^0\}_{n \in \bZ}$ of the operator $L_U(0)$ such that $\curl{\|f_n - f_n^0\|}_{n \in \bZ} \in (\ell^{p})^*$. Since $L_U(0)$ is selfadjoint, then $\{f_n^0\}_{n \in \bZ}$ is an orthonormal basis in $\fH$. If $p=2$ then the proof would be already finished since $\fF$ is $\ell^2$-close to the orthonormal basis $\fF_0$. But as Remark~\ref{rem:c0.bari.diff} shows for $p \in [1,2)$, the $(\ell^p)^*$-closeness to the orthonormal basis is not equivalent to the Bari $(\ell^p)^*$-property.

To this end, let $\fG := \{g_n\}_{n \in \bZ}$ be the system of vectors in $\fH$ that is biorthogonal to the system $\fF$. We need to prove that $\curl{\|f_n - g_n\|}_{n \in \bZ} \in (\ell^{p})^*$. Clearly, $\fG$ is (not normalized) system of root vectors of the adjoint operator $L_U^*(Q)$. Since $L_U(0)$ is self-adjoint then by Lemma~\ref{lem:adjoint} we have $L_U^*(Q) = L_U(Q^*)$. Using Corollary~\ref{cor:every.SRV} in the ``opposite'' direction we can find a normalized system of root vectors $\wt{\fG} := \{\wt{g}_n\}_{n \in \bZ}$ of the operator $L_U(Q^*)$ such that $\curl{\|\wt{g}_n - f_n^0\|}_{n \in \bZ} \in (\ell^{p})^*$. Therefore, $\curl{\|f_n - \wt{g}_n\|}_{n \in \bZ} \in (\ell^{p})^*$.
Since both systems $\fG$ and $\wt{\fG}$ are root vector systems of the operator $L_U^*(Q) = L_U(Q^*)$ and eigenvalues of $L_U(Q^*)$ are asymptotically simple due to Proposition~\ref{prop:Delta.regular.basic}, then vectors $g_n$ and $\wt{g}_n$, $|n| \ge N$, are proportional for some $N \in \bN$. Since $(f_n, g_n) = 1$, $n \in \bZ$, it follows that $\wt{g}_n = (f_n, \wt{g}_n) g_n$, $|n| \ge N$.
Note that if $f, g \in \fH$ and $\|f\| = 1$, then
\begin{multline} \label{eq:f.g-1}
 |(f, g) - 1|^2 = |(f, g)|^2 + 1 - 2 \Re (f, g) \\
 \le \|f\|^2 \|g\|^2 + 1 - 2 \Re (f, g)
 = \|f\|^2 + \|g\|^2 - 2 \Re (f, g) = \|f - g\|^2.
\end{multline}
Since $\|f_n\| = 1$, $n \in \bZ$, then~\eqref{eq:f.g-1} implies that $|(f_n, \wt{g}_n) - 1| \le \|f_n - \wt{g}_n\|$, $n \in \bZ$. Hence for $|n| \ge N$ we have,
\begin{equation} \label{eq:gn-wcgn}
\|\wt{g}_n - g_n\| = \|(f_n, \wt{g}_n) g_n - g_n\| =
|(f_n, \wt{g}_n) - 1| \cdot \|g_n\| \le \|f_n - \wt{g}_n\| \cdot \|g_n\|.
\end{equation}
By the main result of
\cite{LunMal14Dokl,LunMal16JMAA,SavShk14}
the system $\fF = \{f_n\}_{n \in \bZ}$ is a Riesz basis in $\fH$. Hence so is its biorthogonal system $\fG = \{g_n\}_{n \in \bZ}$. This in particular implies that $\|g_n\| \asymp 1$, $n \in \bZ$.
Since $\curl{\|f_n - \wt{g}_n\|}_{n \in \bZ} \in (\ell^{p})^*$ and $\|g_n\| \asymp 1$, $n \in \bZ$, then inequality~\eqref{eq:gn-wcgn} implies that $\curl{\|\wt{g}_n - g_n\|}_{n \in \bZ} \in (\ell^{p})^*$, which in turn implies the desired inclusion $\curl{\|f_n - g_n\|}_{n \in \bZ} \in (\ell^{p})^*$.

\textbf{(ii)} Now assume that some normalized system of root vectors $\fF := \{f_n\}_{n \in \bZ}$ of the operator $L_U(Q)$ is a Bari $(\ell^p)^*$-sequence in $\fH$. By definition $\curl{\|f_n - g_n\|}_{n \in \bZ} \in (\ell^{p})^*$, where $\fG := \{g_n\}_{n \in \bZ}$ is a system biorthogonal to $\fF$ in $\fH$. Clearly, $\fG$ is a system of root vectors of the adjoint operator $L_U^*(Q)$.
By Corollary~\ref{cor:every.SRV} there exists normalized system of root vectors $\fF_0 := \{f_n^0\}_{n \in \bZ}$ of the operator $L_U(0)$ such that $\curl{\|f_n - f_n^0\|}_{n \in \bZ} \in (\ell^{p})^*$. Similarly there exists (possibly not normalized) system of root vectors $\fG_0 := \{g_n^0\}_{n \in \bZ}$ of the operator $L_U^*(0) = L_{U*}(0)$ such that $\curl{\|g_n - g_n^0\|}_{n \in \bZ} \in (\ell^{p})^*$. It is clear, now that $\curl{\|f_n^0 - g_n^0\|}_{n \in \bZ} \in (\ell^{p})^*$.

Let $\wt{\fG}_0 := \{\wt{g}_n^0\}_{n \in \bZ}$ be a system biorthogonal to $\fF_0$. As in part (i), $\wt{\fG}_0$ is a Riesz basis in $\fH$ and $g_n^0 = (f_n^0, g_n^0) \wt{g}_n^0$, $|n| \ge N$. Since $\|f_n^0\| = 1$, $n \in \bZ$, then~\eqref{eq:f.g-1} implies that $|(f_n^0, g_n^0) - 1| \le \|f_n^0 -g_n^0\|$, $n \in \bZ$. Hence
\begin{equation} \label{eq:gn0-wtgn0}
\|g_n^0 - \wt{g}_n^0\| = |(f_n^0, g_n^0) - 1| \cdot
\|\wt{g}_n^0\| \le \|f_n^0 - g_n^0\| \cdot \|\wt{g}_n^0\|, \qquad n \in \bZ.
\end{equation}
Since
$\wt{\fG}_0$ is a Riesz basis, then $\|\wt{g}_n^0\| \asymp 1$, $n \in \bZ$.
Thus, inequality~\eqref{eq:gn0-wtgn0} and inclusion $\curl{\|f_n^0 - g_n^0\|}_{n \in \bZ} \in (\ell^{p})^*$ imply that $\curl{\|f_n^0 - \wt{g}_n^0\|}_{n \in \bZ} \in (\ell^{p})^*$, which means that the normalized root vectors system $\fF_0 = \{f_n^0\}_{n \in \bZ}$ of the operator $L_U(0)$ is a Bari $(\ell^p)^*$-sequence and, in particular, is a Bari $c_0$-sequence. Theorem~\ref{th:crit.c0.bari} now implies that the operator $L_U(0)$ is selfadjoint and finishes the proof.
\end{proof}
\section{Application to a non-canonical string equation}
\label{sec:damped.string}
In this section we show the connection of $2 \times 2$ Dirac type operators with a non-canonical string equation with $u_{xt}$ term, and apply our results on Riesz and Bari basis property.

Consider the following non-canonical hyperbolic equation on a complex-valued function $u(x,t)$ defined for $x \in [0, \len]$ and $t \in [0, \infty)$:
\begin{equation} \label{eq:string}
 u_{tt} - (\beta_1 + \beta_2) u_{xt} + \beta_1 \beta_2 u_{xx} + a_1(x) u_x + a_2(x) u_t = 0,
\end{equation}
with the boundary conditions
\begin{equation} \label{eq:string.cond}
 u(0,t)=0, \qquad h_0 u_x(0,t) + h_1 u_x(\len,t) + h_2 u_t(\len,t)=0, \qquad t \in [0, \infty),
\end{equation}
and initial conditions
\begin{equation} \label{eq:string.init}
 u(x, 0) = u_0(x), \qquad u_t(x, 0) = u_1(x), \qquad x \in [0,\len].
\end{equation}
Here $\beta_1, \beta_2$ are constants and
\begin{equation} \label{eq:b1.b2.a1.a2.h}
 \beta_1 < 0 < \beta_2, \quad a_1, a_2 \in L^1[0,\len],
 \quad h_0, h_1, h_2 \in \bC, \quad |h_1| + |h_2| > 0.
\end{equation}

If $-\beta_1 = \beta_2 = \rho^{-1} > 0$ and $h_0 = 0$, the initial-boundary value problem~\eqref{eq:string}--\eqref{eq:string.init} governs the small vibrations of a string of length $\len$
and density $\rho$ with the presence of a damping coefficient $a_2(x)$;
the string is fixed at the left end ($x=0$), while the right end ($x=\len$) is damped with the coefficient $h_2/h_1 \in \bC \cup \{\infty\}$. Functions $u_0$ and $u_1$ represent the initial position and velocity of the string, respectively.

If $-\beta_1 \ne \beta_2$ one can use linear transform of the variables $x$ and $t$ to reduce it to a classical string equation, but with damping that depends on $t$ and non-classical initial and boundary conditions: initial condition will be on a segment non-parallel to the $x$-axis ($t=0$), while boundary conditions will be on the rays non-parallel to the $t$-axis ($x=0$).

Recall that $W^{1,p}[0, \len]$, $p \ge 1$, denotes the Sobolev space of absolutely continuous functions
with the finite norm
\begin{equation}
 \|f\|_{W^{1,p}[0,\len]}^p
 := \int_0^{\len} \bigl(|f(x)|^p + |f'(x)|^p\bigr) dx < \infty.
\end{equation}
For convenience, we introduce the following notations:
\begin{equation} \label{eq:W1p0}
\wt{W}^{1,p}[0,\len] := \{f \in W^{1,p}[0,\len] :
 f(0) = 0\}, \qquad \wt{H}^1_0[0,\len] := \wt{W}^{1,2}[0,\len],
\end{equation}
where $p \in [1, \infty]$.

The non-canonical initial-boundary value problem~\eqref{eq:string}--\eqref{eq:string.init} of a damped string can be transformed into an abstract Cauchy problem in a Hilbert space $\fH$ of the form
\begin{equation} \label{eq:string.fH.norm}
 \fH :=
 \wt{H}^1_0[0, \len] \times L^2[0, \len], \qquad
\end{equation}
with the inner product
\begin{equation} \label{eq:scal.fH}
 \scal{f,g}_{\fH} := \int_0^\len
 \(f_1'(x) \cdot \ol{g_1'(x)}
 + f_2(x) \cdot \ol{g_2(x)} \)\,dx,
\end{equation}
where $f = \col(f_1, f_2)$, $g = \col(g_1, g_2) \in \fH$.

Now the new representation of the problem~\eqref{eq:string}--\eqref{eq:string.cond} reads as follows:
\begin{equation} \label{eq:Y'=LhY}
 Y'(t) = i \cL Y(t), \quad Y(t) := \binom{u(\cdot,t)}{u_t(\cdot,t)},
 \quad t \ge 0, \qquad Y(0) = \binom{u_0}{u_1},
\end{equation}
where the linear operator $\cL : \dom(\cL) \to \fH$ is defined by
\begin{equation} \label{eq:Lh.def}
 \cL y = \cL \binom{y_1}{y_2} = -i \, \binom{y_2}{-\beta_1
 \beta_2 y_1'' + (\beta_1 + \beta_2) y_2' - a_1 y_1' - a_2 y_2},
\end{equation}
where $y = \col(y_1, y_2) \in \dom(\cL)$, with
\begin{equation} \label{eq:dom.cL}
 \dom(\cL) = \{y = \col(y_1, y_2) \in \fH:
 y_1' \in W^{1,1}[0,\len], \quad
 \cL y \in \fH, \quad
 h_0 y_1'(0) + h_1 y_1'(\len) + h_2 y_2(\len)=0\}.
\end{equation}
It is clear from the definition of $\cL$ and $\dom(\cL)$ that for $y = \col(y_1, y_2) \in \dom(\cL)$ we have: $y_1 \in \wt{W}^{1,1}_0[0,\len]$ and $y_2 \in \wt{H}^1_0[0,\len]$. In particular, $y_1(0)=y_2(0)=0$.

Spectral properties of the operator $\cL$ play important role in the study of stability of solutions of the corresponding string equation. For example, Riesz basis property of the root vectors system of $\cL$ guarantees the exponential stability of the corresponding $C_0$-semigroup. The Riesz basis property and behavior of the spectrum of the operator $\cL$ have been studied in numerous papers (see~\cite{CoxZua94,CoxZua95,Shubov96IEOT,Shubov97,BenRao00,GesHol11,GomRze15,Rzep17} and references therein).

Let us show that the operator $\cL$ is similar to a certain $2 \times 2$ Dirac type operator $L_U(Q)$. Since many spectral properties are preserved under similarity transform, known spectral properties for $2 \times 2$ Dirac type operators will translate to corresponding properties of the dynamic generator $\cL$.

To this end, we need to introduce some notations. Set
\begin{equation} \label{eq:string.B.def}
 B := \diag(b_1,b_2), \qquad b_1 := \beta_1^{-1},
 \quad b_2 := \beta_2^{-1},
\end{equation}
\begin{equation} \label{eq:string.Q.def}
 Q(x) := \frac{i}{b_2-b_1} \begin{pmatrix}
 0 & w(x) \cdot \(b_2^2 a_1(x) + b_2 a_2(x)\) \\
 \frac{-1}{w(x)} \cdot \(b_1^2 a_1(x) + b_1 a_2(x) \) & 0 \end{pmatrix}, \end{equation}
where
\begin{equation} \label{eq:wx.def}
 w(x) := w_1(x) w_2(x),
\end{equation}
\begin{equation} \label{eq:wj.def}
 w_j(x) := \exp\(\frac{b_1 b_2}{b_2-b_1} \int_0^x (b_j a_1(t) + a_2(t)) dt\),
 \qquad x \in [0,\len], \quad j \in \{1,2\}.
\end{equation}
Note, that $w_1(\cdot)$, $w_2(\cdot)$ are well defined and $Q \in L^1([0,\len], \bC^{2 \times 2})$ in view of condition~\eqref{eq:b1.b2.a1.a2.h}. Finally let
\begin{align}
\label{eq:string.U1}
 U_1(y) &:= y_1(0) + y_2(0) = 0, \\
\label{eq:string.U2}
 U_2(y) &:= b_1 h_0 y_1(0) + b_2 h_0 y_2(0) + (b_1 h_1 + h_2) w^{-1}_1(\len)y_1(1)
 + (b_2 h_1 + h_2) w_2(\len) y_2(1) = 0,
\end{align}
be boundary conditions for a Dirac operator $L_U(Q)$. Here $w_1(\cdot)$, $w_2(\cdot)$ are given by~\eqref{eq:wj.def}.
\begin{proposition} \label{prop:Lh.simil}
Operator $\cL$ is similar to the $2 \times 2$ Dirac type operator $L_U(Q)$ with the matrix $B$ given by~\eqref{eq:string.B.def}, the potential matrix $Q(\cdot)$ given by~\eqref{eq:string.Q.def} and boundary conditions $Uy=\{U_1,U_2\}y=0$ given by~\eqref{eq:string.U1}--\eqref{eq:string.U2}.
\end{proposition}
\begin{proof}
We will transform the operator $\cL$ into the desired operator $L_U(Q)$ via series of similarity transformations.

\textbf{Step 1.} Define
\begin{equation}
\cV_0 : \fH \to \LLV{2} \quad\text{as}\quad
\cV_0 y := \binom{y_1'}{y_2}, \quad y = \binom{y_1}{y_2} \in \fH.
\end{equation}
Since $\frac{d}{dx}$ isometrically maps $\wt{H}_0^1[0,\len] = \{f \in W^{1,2}[0,\len] : f(0)=0\}$ onto $L^2[0,\len]$,
then the operator $\cV_0$ is bounded with bounded inverse. It is easy to verify that
\begin{equation} \label{eq:L1def}
 L_1 y := \cV_0 \cL \cV_0^{-1} y
 = -i \,\binom{y_2'}{-\beta_1
 \beta_2 y_1' + (\beta_1 + \beta_2) y_2' - a_1 y_1 - a_2 y_2},
\end{equation}
where
\begin{multline} \label{eq:string.dom.wtL}
 y = \binom{y_1}{y_2} \in \dom(L_1) := \cV_0 \dom(\cL) = \{y \in W^{1,1}([0,\len]; \bC^2) : \\
 L_1y \in L^2([0,\len]; \bC^2), \quad y_2(0) = 0, \ \
 h_0 y_1(0) + h_1 y_1(\len) + h_2 y_2(\len)=0\},
\end{multline}
in view of~\eqref{eq:dom.cL} and definition of $\wt{H}_0^1[0,\len]$. Thus, the operator $\cL$ is similar to the operator $L_1$,
\begin{equation*}
 L_1y = -i B_1 y' + Q_1(x)y,
\end{equation*}
with the domain $\dom(L_1)$ given by~\eqref{eq:string.dom.wtL}, and the matrices $B_1$, $Q_1(\cdot)$, given by
\begin{equation}
 B_1 := \begin{pmatrix} 0 & 1 \\ -\beta_1 \beta_2 & \beta_1 + \beta_2 \\
 \end{pmatrix}, \qquad
 Q_1(x) := \begin{pmatrix} 0 & 0 \\ i a_1(x) & i a_2(x) \end{pmatrix}.
\end{equation}
Note, that $Q_1 \in L^1([0,\len], \bC^{2 \times 2})$ in view of condition~\eqref{eq:b1.b2.a1.a2.h}.

\textbf{Step 2.} Next we diagonalize the matrix $B_1$. To this end let
\begin{equation} \label{eq:string.V1def}
 V_1 := \begin{pmatrix} 1/\beta_1 & 1/\beta_2 \\ 1 & 1 \end{pmatrix}
 = \begin{pmatrix} b_1 & b_2 \\ 1 & 1 \end{pmatrix},
\quad \text{and so} \quad
 V_1^{-1}
 = \frac{1}{b_2-b_1} \begin{pmatrix}
 -1 & b_2 \\ 1 & -b_1 \end{pmatrix},
\end{equation}
where $b_1$ and $b_2$ are defined in~\eqref{eq:string.B.def}. We easily get after straightforward calculations that
\begin{equation} \label{eq:string.V1B1V}
 V_1^{-1} B_1 V_1 = \diag(\beta_1, \beta_2) =
 \diag(b_1^{-1}, b_2^{-1}) = B^{-1},
\end{equation}
\begin{equation} \label{eq:string.V1Q1V}
 V_1^{-1} Q_1(x) V_1 = \frac{i}{b_2-b_1} \begin{pmatrix}
 b_1 b_2 a_1(x) + b_2 a_2(x) &
 b_2^2 a_1(x) + b_2 a_2(x) \\
 - b_1^2 a_1(x) - b_1 a_2(x) &
 - b_1 b_2 a_1(x) - b_1 a_2(x)
 \end{pmatrix} =: Q_2(x), \qquad x \in [0,\len].
\end{equation}
Note, that $Q_2 \in L^1([0,\len], \bC^{2 \times 2})$ in view of condition~\eqref{eq:b1.b2.a1.a2.h}. Introducing bounded operator $\cV_1 : y \to V_1 y$ in $L^2([0,\len]; \bC^2)$, noting that it has a bounded inverse, and taking into account~\eqref{eq:string.V1B1V} and~\eqref{eq:string.V1Q1V}, we obtain
\begin{align}
\nonumber
 L_2 y & := \cV_1^{-1} L_2 \cV_1 y
 = -i V_1^{-1} B_1 V_1 y' + V_1^{-1} Q_1(x) V_1 y \\
\label{eq:L2.def}
 & = -i B^{-1} y' + Q_2(x) y,
 \qquad y \in \cV_1^{-1} \dom(L_1) =: \dom(L_2),
\end{align}
where
\begin{multline} \label{eq:string.dom.L2}
 \dom(L_2) = \{y \in W^{1,1}([0,\len]; \bC^2) : \
 L_2 y \in L^2([0,\len];\bC^2), \
 y_1(0) + y_2(0) = 0, \\
 b_1 h_0 y_1(0) + b_2 h_0 y_2(0) + (b_1 h_1 + h_2) y_1(1)
 + (b_2 h_1 + h_2) y_2(1) = 0\},
\end{multline}
with account of formula~\eqref{eq:string.dom.wtL} for the domain $\dom(L_1)$ and the formula~\eqref{eq:string.V1def} for the matrix $V_1$.

\textbf{Step 3.} On this step we make potential matrix $Q_2$ to be off-diagonal. To this end, Let $\wt{Q}_2$ be a diagonal of $Q_2$, i.e.
$$
 \wt{Q}_2(x) := \frac{i}{b_2-b_1} \begin{pmatrix}
 b_1 b_2 a_1(x) + b_2 a_2(x) & 0 \\
 0 & - b_1 b_2 a_1(x) - b_1 a_2(x)
 \end{pmatrix}.
$$
Let $V_2(\cdot)$ be a solution of the initial value problem
\begin{equation} \label{eq:V2.equ}
 -i B^{-1} V_2'(x) + \wt{Q}_2(x) V_2(x) = 0,
 \qquad V_2(0) = I_2.
\end{equation}
It is easily seen that
\begin{equation} \label{eq:V2.def}
 V_2(x) := \begin{pmatrix} w_1(x) & 0 \\
 0 & w_2^{-1}(x) \end{pmatrix}, \qquad x \in [0,\len],
\end{equation}
where $w_1(\cdot)$, $w_2(\cdot)$ are defined in~\eqref{eq:wj.def}. Let us introduce operator $\cV_2 : y \to V_2(x) y$ in $L^2([0,\len]; \bC^2)$. Since $a_1, a_2 \in L^1[0,\len]$, the operator $\cV_2$ is bounded and has a bounded inverse. Combining relation~\eqref{eq:V2.equ}, definition~\eqref{eq:string.Q.def} of $Q$ and definition~\eqref{eq:wx.def} of $w$, we get
\begin{align}
\nonumber
 L_3 y & := \cV_2^{-1} L_2 \cV_2 y \\
\nonumber
 & = -i [V_2(x)]^{-1} B^{-1} V_2(x) y'
 + [V_2(x)]^{-1} (-i B^{-1} V_2'(x) + Q_2(x) V_2(x)) y \\
\nonumber
 & = -i B^{-1} y'
 + [V_2(x)]^{-1} (Q_2(x) - \wt{Q}_2(x)) V_2(x)) y \\
\label{eq:L3.def}
 & = -i B^{-1} y' + Q(x) y,
 \qquad y \in \cV_2^{-1} \dom(L_2) =: \dom(L_3).
\end{align}
It is clear from the definition of $\cV_2$ that $\dom(L_3)$ coincides with $\dom(L_U(Q))$ defined via~\eqref{eq:string.U1}--\eqref{eq:string.U2}. Hence $L_3 = L_U(Q)$. Combining all the steps of the proof one concludes that $\cL$ is similar to $L_U(Q)$.
\end{proof}
Combining Proposition~\ref{prop:Lh.simil} with our previous results for $2 \times 2$ Dirac type operators we obtain the Riesz basis property and analogous of Bari basis property for the dynamic generator $\cL$ of the non-canonical initial-boundary value problem~\eqref{eq:string}--\eqref{eq:string.init} for a damped string equation. The part (i) of the following result improves known results in the literature on the Riesz basis property for the operator $\cL$ in the case $-\beta_1=\beta_2$, $a_1 \equiv 0$, $h_0=0$ (see~\cite{CoxZua94,CoxZua95,Shubov96IEOT,Shubov97,BenRao00,GesHol11,GomRze15,Rzep17} and references therein). The part (ii) shows the application of one of our main results Theorem~\ref{th:crit.bari}.
\begin{theorem} \label{th:string.riesz}
\textbf{(i)} Let parameters of the damped string equation satisfy relaxed conditions~\eqref{eq:b1.b2.a1.a2.h},
\begin{equation} \label{eq:h2neh1}
\beta_2 h_2 + h_1 \ne 0, \qquad \beta_2 h_2 + h_1 \ne 0,
\end{equation}
and in addition boundary conditions~\eqref{eq:string.U1}--\eqref{eq:string.U2} are strictly regular. Then the system of root vectors of the operator $\cL$ \ \textbf{forms a Riesz basis} in $\fH = \wt{H}^1_0[0,\len] \times L^2[0,\len]$.

\textbf{(ii)} Let in addition $a_1, a_2 \in L^2[0,\len]$. Let also $\cV_0$, $\cV_1$, $\cV_2$ be the operators defined in the steps of the proof of Proposition~\ref{prop:Lh.simil}. Then the system of root vectors of the operator $\cL$ is quadratically close in $\fH$ to a system of the form $\{\cV_0^{-1} \cV_1 \cV_2 e_n\}_{n \in \bZ}$, where $\{e_n\}_{n \in \bZ}$ is an orthonormal basis in $\LL{2}$, if and only if boundary conditions~\eqref{eq:string.U1}--\eqref{eq:string.U2} are self-adjoint, which is equivalent to the condition
\begin{equation} \label{eq:h0=0.b1=-b2}
 h_0 = 0, \quad \beta_1 = -\beta_2, \quad \int_0^{\len} \Im a_2(t) dt = \beta_2 \log \abs{\frac{\beta_2 h_2 + h_1}{\beta_2 h_2 - h_1}}.
\end{equation}
\end{theorem}
\begin{proof}
First, let us transform boundary conditions~\eqref{eq:string.U1}--\eqref{eq:string.U2} to a canonical form~\eqref{eq:cond.canon.intro} assuming condition~\eqref{eq:h2neh1}. For this we multiply the first condition $U_1$ by $b_1 h_0$ and subtract from $U_2$ and then multiple the second condition $U_2$ by $(b_2 h_1 + h_2)^{-1} w_2^{-1}(1)$. Boundary conditions~\eqref{eq:string.U1}--\eqref{eq:string.U2} will take the form
\begin{equation} \label{eq:cond.canon.string}
\begin{cases}
 \wh{U}_{1}(y) = y_1(0) + y_2(0) = 0, \\
 \wh{U}_{2}(y) = d y_2(0) + c y_1(1) + y_2(1) = 0,
\end{cases}
\end{equation}
where
\begin{equation}
d = \frac{(b_2-b_1) h_0}{(b_2 h_1 + h_2) w_2 (1)}, \qquad
c = \frac{b_1 h_1 + h_2}{(b_2 h_1 + h_2) w(\len)}.
\end{equation}
Here $w, w_1, w_2$ are given by~\eqref{eq:wx.def}--\eqref{eq:wj.def}. In particular
\begin{equation} \label{eq:wlen}
 w(\len) := \exp\(\frac{b_1 b_2}{b_2-b_1} \int_0^{\len} ((b_1 + b_2) a_1(t) + 2a_2(t)) dt\).
\end{equation}

\textbf{(i)} Proposition~\ref{prop:Lh.simil} implies that the operator $\cL$ is similar to the operator $L_U(Q)$ with the matrix $B$ given by~\eqref{eq:string.B.def}, the potential matrix $Q(\cdot)$ given by~\eqref{eq:string.Q.def} and boundary conditions $Uy=\{U_1,U_2\}y=0$ given by~\eqref{eq:string.U1}--\eqref{eq:string.U2}. Note that condition~\eqref{eq:h2neh1} implies regularity of boundary conditions~\eqref{eq:string.U1}--\eqref{eq:string.U2}. In addition they are strictly regular by the assumption. Hence operator $L_U(Q)$ has compact resolvent and by Proposition~\ref{prop:Delta.regular.basic} its eigenvalues are asymptotically simple and separated. Moreover, Theorem 1.1 from~\cite{LunMal16JMAA} implies that the system of root vectors of the operator $L_U(Q)$ forms a Riesz basis in $\LLV{2}$. Similarity of $\cL$ and $L_U(Q)$ implies the same properties for $\cL$ in the space $\fH$, which finishes the proof of part (i).

\textbf{(ii)} Since $a_1, a_2 \in L^2[0,\len]$ it follows that $Q \in \LL{2}$. Since boundary conditions~\eqref{eq:cond.canon.string} are strictly regular then by Theorem~\ref{th:crit.bari} (any and every) system of root vectors of the operator $L_U(Q)$ forms a Bari basis in $\LLV{2}$ if only if boundary conditions~\eqref{eq:cond.canon.string} are self-adjoint, which in turn is equivalent to conditions~\eqref{eq:abcd.sa.intro}. Since $a=0$ and $b=1$ then~\eqref{eq:abcd.sa.intro} is equivalent to
\begin{equation}
d=0, \quad b_1 = -b_2, \quad |c| = 1.
\end{equation}
Since $\beta_1 = b_1^{-1}$ and $\beta_2 = b_2^{-1}$, this in turn is equivalent to~\eqref{eq:h0=0.b1=-b2}.

Let us set $\cV := \cV_0^{-1} \cV_1 \cV_2$ and let $\{f_n\}_{n \in \bZ}$ be some system of root vectors of the operator $L_U(Q)$. It follows from the proof of Proposition~\ref{prop:Lh.simil} that $\{\cV f_n\}_{n \in \bZ}$ is a system of root vectors of the operator $\cL$. Hence $\{f_n\}_{n \in \bZ}$ is quadratically close to an orthonormal basis $\{e_n\}_{n \in \bZ}$ in $\LLV{2}$ if and only if $\{\cV f_n\}_{n \in \bZ}$ is quadratically close to $\{\cV e_n\}_{n \in \bZ}$ in $\fH$. This completes the proof.
\end{proof}

\end{document}